\documentclass[fleqn, pdftex]{article} 

\usepackage{amsmath,amssymb,amsthm,pifont,graphics,tikz}

\title{The provability logic of all provability predicates}
\author{Taishi Kurahashi\footnote{Email: kurahashi@people.kobe-u.ac.jp}
\footnote{Graduate School of System Informatics, Kobe University, 1-1 Rokkodai, Nada, Kobe 657-8501, Japan.}}
\date{}

\theoremstyle{plain}
\newtheorem{thm}{Theorem}[section]

\newtheorem{prop}[thm]{Proposition}
\newtheorem{cor}[thm]{Corollary}
\newtheorem{fact}[thm]{Fact}
\newtheorem{prob}[thm]{Problem}
\newtheorem{cl}[thm]{Claim}

\theoremstyle{definition}
\newtheorem{defn}[thm]{Definition}

\newcommand{\PA}{\mathsf{PA}}
\newcommand{\PR}{\mathrm{Pr}}
\newcommand{\Prf}{\mathrm{Prf}}
\newcommand{\Prov}{\mathrm{Prov}}
\newcommand{\Proof}{\mathrm{Proof}}
\newcommand{\Con}{\mathrm{Con}}
\newcommand{\gn}[1]{\ulcorner#1\urcorner}
\newcommand{\D}[1]{\mathbf{D#1}}

\newcommand{\num}{\overline}
\newcommand{\Fml}{\mathrm{Fml}_{\mathcal{L}_A}}

\newcommand{\LA}{\mathcal{L}_A}
\newcommand{\LB}{\mathcal{L}(\Box)}

\newcommand{\GL}{\mathsf{GL}}
\newcommand{\GR}{\mathsf{GR}}
\newcommand{\K}{\mathsf{K}}
\newcommand{\KD}{\mathsf{KD}}
\newcommand{\N}{\mathsf{N}}
\newcommand{\NR}{\mathsf{NR}}
\newcommand{\NF}{\mathsf{N4}}
\newcommand{\NRF}{\mathsf{NR4}}

\newcommand{\Sub}{\mathsf{Sub}}
\newcommand{\MF}{\mathsf{MF}}
\newcommand{\PL}{\mathsf{PL}}

\begin{document}

\maketitle

\begin{abstract}
We prove that the provability logic of all provability predicates is exactly Fitting, Marek, and Truszczy\'nski's pure logic of necessitation $\N$. 
Moreover, we introduce three extensions $\NF$, $\NR$, and $\NRF$ of $\N$ and investigate the arithmetical semantics of these logics. 
In fact, we prove that $\NF$, $\NR$, and $\NRF$ are the provability logics of all provability predicates satisfying the third condition $\D{3}$ of the derivability conditions, all Rosser provability predicates, and all Rosser provability predicates satisfying $\D{3}$, respectively. 
\end{abstract}

\section{Introduction}

Let $T$ be a consistent primitive recursively axiomatized $\LA$-theory containing Peano Arithmetic $\PA$, where $\LA$ is the language of first-order arithmetic. 
G\"odel's second incompleteness theorem states that if a provability predicate $\PR_T(x)$ of $T$ satisfies the following two conditions $\D{2}$ and $\D{3}$, then the consistency statement $\neg \PR_T(\gn{0=1})$ of $T$ cannot be proved in $T$: for any $\LA$-sentences $\varphi$ and $\psi$,  
\begin{description}
	\item [$\D{2}$:] $T \vdash \PR_T(\gn{\varphi \to \psi}) \to \bigl(\PR_T(\gn{\varphi}) \to \PR_T(\gn{\psi}) \bigr)$. 
	\item [$\D{3}$:] $T \vdash \PR_T(\gn{\varphi}) \to \PR_T(\gn{\PR_T(\gn{\varphi})})$. 
\end{description}
In particular, any conventional provability predicate $\Prov_T(x)$ of $T$, which naturally expresses that $x$ is the G\"odel number of a $T$-provable formula, satisfies $\D{2}$ and $\D{3}$. 
Therefore, $T \nvdash \neg \Prov_T(\gn{0=1})$ holds. 

Every provability predicate $\PR_T(x)$ is thought of as a kind of modality, and modal logical study of provability predicates has been developed. 
For each provability predicate $\PR_T(x)$ of $T$, the set of all $T$-verifiable modal formulas under the interpretation where $\Box$ is interpreted by $\PR_T$ is called the \textit{provability logic} of $\PR_T(x)$. 
The most striking result of this study is Solovay's arithmetical completeness theorem \cite{Sol} stating that if $T$ is $\Sigma_1$-sound, then the provability logic of $\Prov_T(x)$ is exactly the G\"odel--L\"ob modal logic $\GL$. 
There has been a wide variety of studies on provability logics such as uniform arithmetical completeness theorem, classification of provability logics, quantified provability logics, bimodal and polymodal provability logics, interpretability logics, and so on. 
See \cite{AB, Boo, JD, Smo} for details of these studies. 

On the other hand, not all provability logics are exactly $\GL$. 
In particular, there exist $\Sigma_1$ provability predicates for which the second incompleteness theorem does not hold. 
A typical example of such a provability predicate is the one that was essentially introduced by Rosser \cite{Ros}. 
Let $\PR_T^{\mathrm{R}}(x)$ be a Rosser provability predicate of $T$ saying that there exists a $T$-proof $y$ of $x$ such that there is no $T$-proof of the negation of $x$ less than $y$. 
It is known that $\neg \PR_T^{\mathrm{R}}(\gn{0=1})$ is provable in $\PA$. 
Hence, the provability logic of $\PR_T^{\mathrm{R}}(x)$ is completely different from $\GL$ because it contains the modal formula $\neg \Box \bot$ that is inconsistent with $\GL$. 
Also, by the proof of the second incompleteness theorem, $\PR_T^{\mathrm{R}}(x)$ does not satisfy at least one of the conditions $\D{2}$ and $\D{3}$. 
And, it has been shown that whether or not $\PR_T^{\mathrm{R}}(x)$ satisfies either $\D{2}$ or $\D{3}$ depends on the details of construction of $\PR_T^{\mathrm{R}}(x)$.
In other words, whether the corresponding provability logic contains either $\Box (A \to B) \to (\Box A \to \Box B)$ or $\Box A \to \Box \Box A$ depends on the choice of $\PR_T^{\mathrm{R}}(x)$.
Indeed, Bernardi and Montagna \cite{BM} and Arai \cite{Ara} proved that there exists a Rosser provability predicate $\PR_T^{\mathrm{R}}(x)$ of $T$ satisfying $\D{2}$, and hence such a predicate does not satisfy $\D{3}$. 
Arai also proved the existence of a Rosser provability predicate satisfying $\D{3}$. 
Observe that the provability logics of Rosser provability predicates satisfying $\D{2}$ contain the normal modal logic $\KD$. 
The author proved in \cite{Kur20} that there exists a Rosser provability predicate whose provability logic is exactly $\KD$.

Provability predicates that are not $\Sigma_1$ whose provability logics are different from $\GL$ have also been studied. 
For example, the author proved in \cite{Kur18_1} that there exists a $\Sigma_2$ provability predicate of $T$ whose provability logic is exactly the weakest normal modal logic $\K$. 
Also, for several normal modal logics, the existence of corresponding $\Sigma_2$ provability predicates has been shown (cf.~\cite{Kur18_2,Sha94,Vis}).

In previous studies, all unimodal logics that have been considered as provability logics are normal, that is, containing the logic $\K$. 
Provability predicates corresponding to such logics always satisfy the condition $\D{2}$. 
In general, however, not all provability predicates satisfy $\D{2}$. 
For example, Rosser provability predicates satisfying $\D{3}$, whose existence was proved by Arai, do not satisfy $\D{2}$. 
The provability logics corresponding to such predicates are non-normal. 

In the present paper, we discuss non-normal provability logics through the following questions: 
\begin{itemize}
	\item [Q1] What is the intersection of all provability logics, that is, the provability logic of all provability predicates?
	\item [Q2] What is the provability logic of all Rosser provability predicates?
\end{itemize}

The property common to all provability predicates is ``$T \vdash \varphi \Rightarrow T \vdash \PR_T(\gn{\varphi})$'' that corresponds to the closure under the Necessitation rule $\dfrac{A}{\Box A}$, and presumably no other. 
The non-normal modal logic $\N$, obtained by adding Necessitation $\dfrac{A}{\Box A}$ as an inference rule to classical propositional logic, was introduced by Fitting, Marek, and Truszczy\'nski \cite{FMT}. 
In that paper, $\N$ is called the \textit{pure logic of necessitation}.
This logic $\N$ is our candidate for the answer to Q1, but a problem arises. 
The usual proof of Solovay's theorem is to embed Kripke models into arithmetic, and similar techniques have been used in the proofs of the previously mentioned results for various normal modal logics. 
On the other hand, since the logic $\N$ is not a normal modal logic, $\N$ does not have Kripke semantics. 
However, Fitting, Marek, and Truszczy\'nski introduced a Kripke-like relational semantics corresponding to $\N$, and the soundness, completeness, and finite frame property of $\N$ with respect to that semantics were proved. 
Then, we can attempt to apply Solovay's method to that semantics.
Indeed, in Section \ref{Sec:N}, we prove that $\N$ is exactly the provability logic of all provability predicates. 
This is the answer to Q1. 
Moreover, we actually prove more: There exists a $\Sigma_1$ provability predicate of $T$ whose provability logic is exactly $\N$. 

Shavrukov \cite{Sha91} introduced the bimodal logic $\GR$ of the standard and Rosser provability predicates and proved the arithmetical completeness theorem for $\GR$. 
Thus, the unimodal fragment $L^R$ of $\GR$ is the answer to Q2, but no specific axiomatization of the logic $L^R$ has been made so far. 
In Section \ref{Sec:Compl}, we introduce the logic $\NR$ that is obtained from $\N$ by adding the inference rule $\dfrac{\neg A}{\neg \Box A}$. 
We prove the finite frame property of $\NR$ with respect to Fitting, Marek, and Truszczy\'nski's semantics. 
In Section \ref{Sec:NR}, we prove that $\NR$ precisely coincides with $L^R$ and that $\NR$ is exactly the provability logic of all Rosser provability predicates. 
This is the answer to Q2. 

Furthermore, in the present paper, we deal with provability predicates satisfying the condition $\D{3}$. 
In Section \ref{Sec:Compl}, we also introduce the logics $\NF$ and $\NRF$ that are obtained from $\N$ and $\NR$ by adding the axiom scheme $\Box A \to \Box \Box A$, respectively. 
We prove the finite frame property of $\NF$ and $\NRF$. 
Also, in Sections \ref{Sec:NF} and \ref{Sec:NRF}, we prove that $\NF$ and $\NRF$ are exactly the provability logics of all provability predicates satisfying $\D{3}$ and all Rosser provability predicates satisfying $\D{3}$, respectively.

In Appendix 1, as a related topic, we prove the existence of a $\Sigma_1$ provability predicate whose provability logic is exactly $\K$. 
In Appendix 2, we prove the interchangeability of $\Box$ and $\Diamond$ in $\NR$. 
As a continuation of the present paper, in \cite{KK}, non-normal provability logics closed under the rule $\dfrac{A \to B}{\Box A \to \Box B}$ are investigated.

\section{Preliminaries}\label{Sec:Pre}

Let $\LA$ be the language of first-order arithmetic. 
Also, let $\omega$ be the set of all natural numbers. 
We fix a natural G\"odel numbering such that if $\psi$ is a proper subformula of $\varphi$, then the G\"odel number of $\psi$ is smaller than that of $\varphi$. 
We may also assume that $0$ is not the G\"odel number of anything. 
Let $\{\xi_t\}_{t \in \omega}$ be the repetition-free primitive recursive enumeration of all $\LA$-formulas arranged in ascending order of G\"odel numbers. 
That is, if $\xi_s$ is a proper subformula of $\xi_u$, then $s < u$. 
For each $n \in \omega$, let $\num{n}$ be the numeral for $n$. 
We assume that for any natural number $n$ and any formula $\varphi(x)$ having $x$ as a free variable, the G\"odel number of $\varphi(\num{n})$ is larger than $n$. 
For each $\LA$-formula $\varphi$, let $\gn{\varphi}$ be the numeral for the G\"odel number of $\varphi$. 
For any two formulas $\varphi$ and $\psi$, $\varphi \equiv \psi$ means that $\varphi$ and $\psi$ are syntactically identical. 

Throughout the present paper, $T$ always denotes a consistent primitive recursively axiomatized $\LA$-theory containing Peano Arithmetic $\PA$. 
We say that an $\LA$-formula $\PR_T(x)$ is a \textit{provability predicate} of $T$ if for any $\LA$-formula $\varphi$, we have that $T \vdash \varphi$ if and only if $\PA \vdash \PR_T(\gn{\varphi})$.  
In his proof of the incompleteness theorems, G\"odel constructed a primitive recursive proof predicate $\Proof_T(x, y)$ of $T$ naturally saying that $y$ is the G\"odel number of a $T$-proof of an $\LA$-formula whose G\"odel number is $x$. 
The $\Sigma_1$ formula $\Prov_T(x)$ defined by $\exists y \Proof_T(x, y)$ is a standard provability predicate of $T$. 
Then, it is known that $\Prov_T(x)$ satisfies the conditions $\D{2}$ and $\D{3}$ given in the introduction. 
We naturally assume that $\PA \vdash \forall x \forall y (\Proof_T(x, y) \to x \leq y)$. 

Here, we comment on our definition of provability predicates. 
We say that a formula $\PR_T(x)$ satisfies the \textit{Kreisel Condition} for $T$ if for any $\LA$-formula $\varphi$, $T \vdash \varphi$ if and only if $T \vdash \PR_T(\gn{\varphi})$ (cf.~Visser \cite{Vis21}). 
In some cases, it is convenient to define provability predicates of $T$ as formulas satisfying the Kreisel Condition for $T$. 
However, since the standard provability predicate $\Prov_T(x)$ of $T$ does not generally satisfy the Kreisel Condition for $T$, we prefer to adopt our definition of provability predicates.
Moreover, a $\Sigma_1$ formula $\PR_T(x)$ being a provability predicate of $T$ is equivalent to a natural condition that for any $\LA$-formula $\varphi$, $T \vdash \varphi$ if and only if $\mathbb{N} \models \PR_T(\gn{\varphi})$, where $\mathbb{N}$ is the standard model of arithmetic. 

We say that a $\Sigma_1$ formula $\PR_T^{\mathrm{R}}(x)$ is a \textit{Rosser provability predicate} of $T$ if there exists a primitive recursive formula $\Prf_T(x, y)$ satisfying the following three conditions:
\begin{enumerate}
	\item For any $\LA$-formula $\varphi$ and $n \in \omega$, $\PA \vdash \Proof_T(\gn{\varphi}, \num{n}) \leftrightarrow \Prf_T(\gn{\varphi}, \num{n})$. 
	\item $\PA \vdash \forall x \Bigl(\Fml(x) \to \bigl(\Prov_T(x) \leftrightarrow \exists y \Prf_T(x, y) \bigr) \Bigr)$, where $\Fml(x)$ is a primitive recursive formula naturally expressing that $x$ is the G\"odel number of an $\LA$-formula. 
	\item $\PR_T^{\mathrm{R}}(x)$ is of the form $\exists y \bigl(\Fml(x) \land \Prf_T(x, y) \land \forall z < y \, \neg \Prf_T(\dot{\neg}(x), z) \bigr)$, where $\dot{\neg}(x)$ is a primitive recursive term corresponding to a primitive recursive function calculating the G\"odel number of $\neg \varphi$ from that of $\varphi$. 
\end{enumerate}
It is known that each Rosser provability predicate of $T$ is in fact a $\Sigma_1$ provability predicate of $T$. 
The idea of witness comparison, which is technically very important, is behind Rosser provability predicates (see~\cite{GS,Lin}).
Based on the witness comparison argument, it is shown that for any Rosser provability predicate $\PR_T^{\mathrm{R}}(x)$ of $T$ and any $\LA$-formula $\varphi$, if $T \vdash \neg \varphi$, then $\PA \vdash \neg \PR_T^{\mathrm{R}}(\gn{\varphi})$.

The language $\LB$ of modal propositional logic consists of propositional variables, the logical constant $\bot$, propositional connectives $\neg, \land, \lor, \to$, and the modal operator $\Box$. 
Let $\MF$ be the set of all $\LB$-formulas. 
The axioms of the modal logic $\K$ are all propositional tautologies in $\LB$ and the axiom scheme $\Box(A \to B) \to (\Box A \to \Box B)$. 
The inference rules of $\K$ are Modus Ponens (\textsc{MP}) $\dfrac{A \quad A \to B}{B}$ and Necessitation (\textsc{Nec}) $\dfrac{A}{\Box A}$.
The modal logics $\KD$, $\mathsf{K4}$, and $\GL$ are obtained from $\K$ by adding the axiom schemata $\neg \Box \bot$, $\Box A \to \Box \Box A$, and $\Box(\Box A \to A) \to \Box A$, respectively. 
When we interpret $\Box$ by a provability predicate $\PR_T(x)$, then the axiom schemata $\Box(A \to B) \to (\Box A \to \Box B)$, $\Box A \to \Box \Box A$, and $\Box(\Box A \to A) \to \Box A$ correspond to $\D{2}$, $\D{3}$, and formalized L\"ob's theorem, respectively. 

To state these correspondences precisely, we introduce the notion of arithmetical interpretations. 
For each provability predicate $\PR_T(x)$ of $T$, a mapping $f$ from $\MF$ to a set of $\LA$-sentences is called an \textit{arithmetical interpretation based on $\PR_T(x)$} if it satisfies the following conditions: 
\begin{enumerate}
	\item $f(\bot)$ is $0=1$, 
	\item $f(\neg A)$ is $\neg f(A)$, 
	\item $f(A \circ B)$ is $f(A) \circ f(B)$ for $\circ \in \{\land, \lor, \to\}$, 
	\item $f(\Box A)$ is $\PR_T(\gn{f(A)})$. 
\end{enumerate}
Let $\PL(\PR_T)$ be the set of all $\LB$-formulas $A$ satisfying that for any arithmetical interpretation $f$ based on $\PR_T(x)$, $T \vdash f(A)$. 
The set $\PL(\PR_T)$ is called the \textit{provability logic} of $\PR_T(x)$. 

It is obvious that for any provability predicate $\PR_T(x)$ of $T$, $\PL(\PR_T)$ is closed under \textsc{Nec}. 
If $\PR_T(x)$ satisfies $\D{2}$, then $\PL(\PR_T)$ contains the logic $\K$, that is, $\PL(\PR_T)$ is a normal modal logic. 

The study of provability logics can be approached from two directions corresponding to the following two problems, respectively. 

\begin{prob}\label{MProb1}
For each provability predicate $\PR_T(x)$ of $T$, how is $\PL(\PR_T)$ axiomatized and what properties does it have? 
\end{prob}

\begin{prob}\label{MProb2}
For which modal logics $L$ is there a provability predicate $\PR_T(x)$ such that $L = \PL(\PR_T)$?
\end{prob}

The most striking result concerning the first problem is Solovay's arithmetical completeness theorem \cite{Sol}. 
It states that if $T$ is $\Sigma_1$-sound, then $\PL(\Prov_T)$ is exactly $\GL$. 
Visser \cite{Vis} proved that if $T$ is not $\Sigma_1$-sound, then $\PL(\Prov_T)$ is either $\GL$ or $\GL + \Box^n \bot$ for some $n \geq 1$. 
As an interesting example regarding the first problem, we present here Shavrukov's result \cite{Sha94}.  
Let $\Pr_{\PA}^{\mathrm{Sh}}(x)$ be the $\Sigma_2$ provability predicate $\exists y \bigl(\Prov_{\mathbf{I\Sigma}_y}(x) \land \neg \Prov_{\mathbf{I\Sigma}_y}(\gn{0=1}) \bigr)$ of $\PA$, which was essentially introduced by Smory\'nski \cite{Smo89}. 
Shavrukov proved that $\PL(\PR_{\PA}^{\mathrm{Sh}})$ is the logic $\KD + (\Box A \to \Box((\Box B \to B) \lor \Box A))$. 

For the second problem, the following results have been obtained by previous studies. 

\begin{itemize}
	\item (Kurahashi \cite{Kur20}) There exists a Rosser provability predicate $\PR_T^{\mathrm{R}}(x)$ of $T$ such that $\PL(\PR_T^{\mathrm{R}}) = \KD$. 
	\item (Kurahashi \cite{Kur18_1}) There exists a $\Sigma_2$ provability predicate $\PR_T(x)$ of $T$ such that $\PL(\PR_T) = \K$. 
	\item (Kurahashi \cite{Kur18_2}) For each $n \geq 2$, there exists a $\Sigma_2$ provability predicate $\PR_T(x)$ of $T$ such that $\PL(\PR_T) = \K + (\Box(\Box^n A \to A) \to \Box A)$. 
	\item (Montague \cite{Mon63}) For any provability predicate $\PR_T(x)$ of $T$, $\PL(\PR_T) \nsupseteq \mathsf{KT}$ ($= \K + (\Box A \to A)$). 
	\item (L\"ob \cite{Lob}) For any provability predicate $\PR_T(x)$ of $T$, $\PL(\PR_T) \neq \mathsf{K4}$. 
	\item (Kurahashi \cite{Kur18_2}) For any provability predicate $\PR_T(x)$ of $T$, if $T$ does not prove $\PR_T(\gn{0=1})$, then $\PL(\PR_T) \nsupseteq \mathsf{KB}$ ($= \K + (A \to \Box \Diamond A)$) and $\PL(\PR_T) \nsupseteq \mathsf{K5}$ ($= \K + (\Diamond A \to \Box \Diamond A)$). 
\end{itemize}

All of the above results are for normal modal logics. 
On the other hand, there is a result concerning a non-normal modal logic. 
Shavrukov \cite{Sha91} introduced the bimodal logic $\GR$ of the standard and Rosser provability predicates. 
Let $\mathcal{L}(\Box, \blacksquare)$ be the language of modal propositional logic equipped with an additional modal operator $\blacksquare$. 
The axiom schemata of $\GR^-$ are as follows: 
\begin{enumerate}
	\item Those of $\GL$ for $\Box$, 
	\item $\blacksquare A \to \Box A$, 
	\item $\Box A \to \Box \blacksquare A$, 
	\item $\Box A \to (\Box \bot \lor \blacksquare A)$, 
	\item $\Box \neg A \to \Box \neg \blacksquare A$. 
\end{enumerate}
The inference rules of $\GR^-$ are \textsc{MP} and \textsc{Nec} for $\Box$. 
The logic $\GR$ is obtained from $\GR^-$ by adding the rule $\dfrac{\Box A}{A}$. 
The studies of $\GR^-$ and $\GR$ presented in \cite{Sha91} were based on those of the logics $\mathsf{R}^-$ and $\mathsf{R}$ developed by Guaspari and Solovay \cite{GS}. 
The logic $\GR$ can be embedded into $\GR^-$ as follows. 

\begin{thm}[{\cite[Corollary 1.10]{Sha91}}]\label{GRm}
For any $\mathcal{L}(\Box, \blacksquare)$-formula $A$, $\GR \vdash A$ if and only if $\GR^- \vdash \Box A$. 
\end{thm}

A \textit{bimodal arithmetical interpretation $f$ based on $(\PR_T, \PR_T^{\mathrm{R}})$} is an arithmetical interpretation based on $\PR_T(x)$ such that $f(\blacksquare A)$ is $\PR_T^{\mathrm{R}}(\gn{f(A)})$. 
Shavrukov proved the following arithmetical soundness and completeness theorems. 

\begin{thm}[The arithmetical soundness theorem of $\GR$ {\cite[Lemma 2.5]{Sha91}}]
For any Rosser provability predicate $\PR_T^{\mathrm{R}}(x)$ of $T$, any bimodal arithmetical interpretation $f$ based on $(\Prov_T, \PR_T^{\mathrm{R}})$, and any $\mathcal{L}(\Box, \blacksquare)$-formula $A$, if $\GR \vdash A$, then $\PA \vdash f(A)$. 
\end{thm}

\begin{thm}[The uniform arithmetical completeness theorem for $\GR$ {\cite[Theorem 3.1]{Sha91}}]
Suppose that $T$ is $\Sigma_1$-sound. 
Then, there exist a Rosser provability predicate $\PR_T^{\mathrm{R}}(x)$ of $T$ and a bimodal arithmetical interpretation $f$ based on $(\Prov_T, \PR_T^{\mathrm{R}})$ such that for any $\mathcal{L}(\Box, \blacksquare)$-formula $A$, $\GR \vdash A$ if and only if $T \vdash f(A)$. 
\end{thm}

Let $L^{\mathrm{R}}$ be the unimodal logic obtained by replacing all $\blacksquare$ in the $\Box$-free fragment of $\GR$ by $\Box$. 
The following corollary follows from Shavrukov's theorems. 

\begin{cor}\label{Cor:LR}
If $T$ is $\Sigma_1$-sound, then 
\[
	L^{\mathrm{R}} = \bigcap \{\PL(\PR_T^{\mathrm{R}}) \mid \PR_T^{\mathrm{R}}(x)\ \text{is a Rosser provability predicate of}\ T\}.
\] 
Furthermore, there exists a Rosser provability predicate $\PR_T^{\mathrm{R}}(x)$ of $T$ such that $L^{\mathrm{R}} = \PL(\PR_T^{\mathrm{R}})$. 
\end{cor}

Corollary \ref{Cor:LR} states that $L^{\mathrm{R}}$ is the provability logic of all Rosser provability predicates. 
The logic $L^{\mathrm{R}}$ is a non-normal modal logic because there are Rosser provability predicates that do not satisfy $\D{2}$. 
However, since no specific axiomatization for $L^{\mathrm{R}}$ is obtained, Corollary \ref{Cor:LR} is not sufficient for us in view of Problem \ref{MProb1}.
In this context, our purpose in the present paper is to axiomatize the following four logics: 
\begin{enumerate}
	\item $\bigcap \{\PL(\PR_T) \mid \PR_T(x)\ \text{is a provability predicate of}\ T\}$,
	\item $\bigcap \{\PL(\PR_T) \mid \PR_T(x)\ \text{is a provability predicate of}\ T\ \text{satisfying}\ \D{3}\}$,
	\item $\bigcap \{\PL(\PR_T^{\mathrm{R}}) \mid \PR_T^{\mathrm{R}}(x)\ \text{is a Rosser provability predicate of}\ T\}$,
	\item $\bigcap \{\PL(\PR_T^{\mathrm{R}}) \mid \PR_T^{\mathrm{R}}(x)\ \text{is a Rosser provability predicate of}\ T\ \text{satisfying}\ \D{3}\}$. 
\end{enumerate}
In the next section, we introduce the logics $\N$, $\NF$, $\NR$, and $\NRF$ which are candidates for axiomatizations of these logics. 
We study these logics from the point of view of Problems \ref{MProb1} and \ref{MProb2}.

\section{The logic $\N$ and its extensions}\label{Sec:Compl}

For any provability predicate $\PR_T(x)$, the provability logic $\PL(\PR_T)$ is closed under \textsc{Nec}. 
Thus, the provability logic 
\[
	\bigcap \{\PL(\PR_T) \mid \PR_T(x)\ \text{is a provability predicate of}\ T\}
\]
of all provability predicates is also closed under \textsc{Nec}. 
On the other hand, there seems to be no other non-trivial modal logical principle that is common to all provability predicates. 
Our candidate for the axiomatization of the provability logic of all provability predicates is the pure logic of necessitation $\N$ that was introduced by Fitting, Marek, and Truszczy\'nski \cite{FMT}. 
The axioms of $\N$ are propositional tautologies in the language $\LB$ and the inference rules of $\N$ are \textsc{MP} and \textsc{Nec}. 

Fitting, Marek, and Truszczy\'nski introduced the following natural relational semantics for $\N$. 

\begin{defn}[$\N$-frames]\leavevmode
\begin{itemize}
	\item We say that a tuple $(W, \{\prec_B\}_{B \in \MF})$ is an \textit{$\N$-frame} if $W$ is a non-empty set and for each $B \in \MF$, $\prec_B$ is a binary relation on $W$.  
	\item We say that a triple $(W, \{\prec_B\}_{B \in \MF}, \Vdash)$ is an \textit{$\N$-model} if $(W, \{\prec_B\}_{B \in \MF})$ is an $\N$-frame and $\Vdash$ is a satisfaction relation on $W \times \MF$ fulfilling the usual conditions for propositional connectives and
	\[
		x \Vdash \Box B \iff \forall y \in W(x \prec_B y \Rightarrow y \Vdash B).
	\]
	\item A formula $A$ is \textit{valid} in an $\N$-model $(W, \{\prec_B\}_{B \in \MF}, \Vdash)$ if for any $x \in W$, $x \Vdash A$.
	\item A formula $A$ is \textit{valid} in an $\N$-frame $(W, \{\prec_B\}_{B \in \MF})$ if $A$ is valid in any $\N$-model $(W, \{\prec_B\}_{B \in \MF}, \Vdash)$ based on $(W, \{\prec_B\}_{B \in \MF})$.
	\item We say that an $\N$-frame $(W, \{\prec_B\}_{B \in \MF})$ is \textit{finite} if $W$ is a finite set. 
\end{itemize}
\end{defn}

Fitting, Marek, and Truszczy\'nski proved that $\N$ is sound and complete and has the finite frame property with respect to this semantics. 

\begin{fact}[Fitting, Marek, and Truszczy\'nski {\cite[Theorems 3.6 and 4.10]{FMT}}]\label{Fact1}
For any $A \in \MF$, the following are equivalent:
\begin{enumerate}
	\item $\N \vdash A$. 
	\item $A$ is valid in all $\N$-frames. 
	\item $A$ is valid in all finite $\N$-frames. 
\end{enumerate}
\end{fact}

Each $\N$-model has infinitely many binary relations $\{\prec_B\}_{B \in \MF}$, but the truth of each $\LB$-formula in each element of the model is determined by referring to only a finite number of those relations. 
Let $\Sub(A)$ be the set of all subformulas of $A \in \MF$. 

\begin{fact}[Fitting, Marek, and Truszczy\'nski {\cite[Theorem 4.11]{FMT}}]\label{Fact2}
Let $A \in \MF$. 
Let $(W, \{\prec_B\}_{B \in \MF}, \Vdash)$ and $(W, \{\prec^*_B\}_{B \in \MF}, \Vdash^*)$ be any $\N$-models satisfying the following two conditions:
\begin{enumerate}
	\item For each $x \in W$ and each propositional variable $p \in \Sub(A)$, we have that $x \Vdash p \iff x \Vdash^* p$; 
	\item For each $\Box B \in \Sub(A)$, $\prec_B = \prec^*_B$. 
\end{enumerate}
Then, for every $x \in W$, $x \Vdash A \iff x \Vdash^* A$. 
\end{fact}

We introduce the three extensions $\NR$, $\NF$, and $\NRF$ of $\N$. 

\begin{defn}\leavevmode
\begin{itemize}
	\item The logic $\NR$ is obtained from $\N$ by adding the inference rule $\dfrac{\neg B}{\neg \Box B}$. 
	\item The logics $\NF$ and $\NRF$ are obtained from $\N$ and $\NR$, respectively, by adding the axiom scheme $\Box B \to \Box \Box B$. 
\end{itemize}
\end{defn}

We call the rule $\dfrac{\neg B}{\neg \Box B}$ the Rosser rule (\textsc{Ros}). 
Before proving the completeness theorems for these logics, we show that the validity of these logics is related to some appropriate conditions on $\N$-frames. 

\begin{defn}
Let $A \in \MF$ and $\Gamma \subseteq \MF$. 
Let $\mathcal{F} = (W, \{\prec_B\}_{B \in \MF})$ be any $\N$-frame. 
\begin{itemize}
	\item $\mathcal{F}$ is called \textit{$A$-serial} if for every $x \in W$, there exists a $y \in W$ such that $x \prec_A y$. 
	\item $\mathcal{F}$ is said to be \textit{$\Gamma$-serial} if $\mathcal{F}$ is $A$-serial for every $\Box A \in \Gamma$.  
	\item $\mathcal{F}$ is called \textit{serial} if $\mathcal{F}$ is $\MF$-serial. 
\end{itemize}
\end{defn} 

\begin{prop}
Let $A \in \MF$ and $\mathcal{M} = (W, \{\prec_B\}_{B \in \MF}, \Vdash)$ be any $\N$-model. 
Suppose that the $\N$-frame $\mathcal{F} = (W, \{\prec_B\}_{B \in \MF})$ is $A$-serial. 
If $\neg A$ is valid in $\mathcal{M}$, then $\neg \Box A$ is also valid in $\mathcal{M}$. 
\end{prop}
\begin{proof}
Suppose that $\mathcal{F}$ is $A$-serial and $\neg A$ is valid in $\mathcal{M}$. 
Let $x \in W$ be any element. 
Since $\mathcal{F}$ is $A$-serial, there exists a $y \in W$ such that $x \prec_A y$. 
Since $\neg A$ is valid in $\mathcal{M}$, we have $y \Vdash \neg A$. 
Thus, $x \Vdash \neg \Box A$. 
Therefore, $\neg \Box A$ is valid in $\mathcal{M}$. 
\end{proof}

\begin{cor}\label{Cor:NR}
Let $A \in \MF$. 
If $\NR \vdash A$, then $A$ is valid in all serial $\N$-frames. 
\end{cor}

\begin{defn}
Let $A \in \MF$ and $\Gamma \subseteq \MF$. 
Let $\mathcal{F} = (W, \{\prec_B\}_{B \in \MF})$ be any $\N$-frame. 
\begin{itemize}
	\item $\mathcal{F}$ is called \textit{$A$-transitive} if for every $x, y, z \in W$, if $x \prec_{\Box A} y$ and $y \prec_A z$, then $x \prec_A z$. 
	\item $\mathcal{F}$ is said to be \textit{$\Gamma$-transitive} if $\mathcal{F}$ is $A$-transitive for every $\Box \Box A \in \Gamma$.  
	\item $\mathcal{F}$ is called \textit{transitive} if $\mathcal{F}$ is $\MF$-transitive. 
\end{itemize}
\end{defn}

\begin{prop}\label{Prop:trans}
Let $A \in \MF$ and $\mathcal{F} = (W, \{\prec_B\}_{B \in \MF})$ be any $\N$-frame. 
If $\mathcal{F}$ is $A$-transitive, then $\Box A \to \Box \Box A$ is valid in $\mathcal{F}$. 
\end{prop}
\begin{proof}
Suppose that $\mathcal{F}$ is $A$-transitive. 
Let $(\mathcal{F}, \Vdash)$ be any $\N$-model based on $\mathcal{F}$. 
Let $x \in W$ be any element with $x \Vdash \Box A$. 
Let $y, z \in W$ be such that $x \prec_{\Box A} y$ and $y \prec_A z$. 
Since $\mathcal{F}$ is $A$-transitive, we have $x \prec_A z$. 
Then, $z \Vdash A$. 
Since $z$ is an arbitrary element with $y \prec_A z$, we have $y \Vdash \Box A$. 
Also, we obtain $x \Vdash \Box \Box A$. 
We conclude that $\Box A \to \Box \Box A$ is valid in $\mathcal{F}$. 
\end{proof}

\begin{cor}\label{Cor:NF}
Let $A \in \MF$. 
\begin{enumerate}
	\item If $\NF \vdash A$, then $A$ is valid in all transitive $\N$-frames. 
	\item If $\NRF \vdash A$, then $A$ is valid in all transitive and serial $\N$-frames. 
\end{enumerate}
\end{cor}

Unlike the case of Kripke frames, the validity of $\Box A \to \Box \Box A$ in an $\N$-frame is not equivalent to the $A$-transitivity in general. 

\begin{prop}
There exists an $\N$-frame $\mathcal{F}$ satisfying the following conditions: 
\begin{enumerate}
	\item $\Box B \to \Box \Box B$ is valid in $\mathcal{F}$ for all $B \in \MF$. 
	\item For any $B, C_0, \ldots, C_{k-1} \in \MF$, if $\N \vdash \Box C_0 \land \cdots \land \Box C_{k-1} \to B$, then $\mathcal{F}$ is not $B$-transitive.
\end{enumerate} 
\end{prop}
\begin{proof}
Let $\mathcal{F} = (W, \{\prec_B\}_{B \in \MF})$ be the $\N$-frame defined as follows: 
\begin{itemize}
	\item $W : = \{a, b, c\}$, 
	\item \begin{itemize}
		\item If $\N \vdash \Box C_0 \land \cdots \land \Box C_{k-1} \to B$ for some $k$ and $C_0, \ldots, C_{k-1} \in \MF$, then $\prec_{B} : = \{(a, b), (b, c)\}$. 
		\item Otherwise, $\prec_{B} : = \{(a, b), (a, c), (b, c)\}$. 
	\end{itemize}
\end{itemize}

The second clause of the proposition immediately follows from the definition because $a \prec_{\Box B} b$ and $b \prec_B c$ for each $B \in \MF$. 
It suffices to show that $\Box B \to \Box \Box B$ is valid in $\mathcal{F}$ for all $B \in \MF$. 
Assume that $\N \nvdash \Box C_0 \land \cdots \land \Box C_{k-1} \to B$ for any $k$ and $C_0, \ldots, C_{k-1} \in \MF$. 
Then, it is shown that $\mathcal{F}$ is $B$-transitive. 
By Proposition \ref{Prop:trans}, $\Box B \to \Box \Box B$ is valid in $\mathcal{F}$. 

So, we may assume that $\N \vdash \Box C_0 \land \cdots \land \Box C_{k-1} \to B$ for some $C_0, \ldots, C_{k-1} \in \MF$. 
In this case, $\prec_{B} = \{(a, b), (b, c)\}$. 
Let $(\mathcal{F}, \Vdash)$ be any $\N$-model based on $\mathcal{F}$. 
For each $C \in \MF$, we have $c \Vdash \Box C$ because there is no $w \in W$ such that $c \prec_C w$. 
So, we have $c \Vdash \Box B$ and $c \Vdash \Box \Box B$. 
Moreover, we have $c \Vdash B$ since $c \Vdash \Box C_i$ for $i < k$. 
Then, we obtain $b \Vdash \Box B$ and $b \Vdash \Box \Box B$. 
Also, we obtain $a \Vdash \Box \Box B$. 
We conclude that $\Box B \to \Box \Box B$ is valid in $(\mathcal{F}, \Vdash)$. 
\end{proof}

We prove the completeness and finite frame property of the logics $\NR$, $\NF$, and $\NRF$. 
I also simultaneously give an alternative proof of Fact \ref{Fact1}.

\begin{thm}[The completeness and finite frame property of $\NR$]\label{Thm:complNR}
For any $A \in \MF$, the following are equivalent: 
\begin{enumerate}
	\item $\NR \vdash A$. 
	\item $A$ is valid in all serial $\N$-frames. 
	\item $A$ is valid in all finite serial $\N$-frames. 
	\item $A$ is valid in all finite $\Sub(A)$-serial $\N$-frames. 
\end{enumerate}
\end{thm}

\begin{thm}[The completeness and finite frame property of $\NF$]\label{Thm:complNF}
For any $A \in \MF$, the following are equivalent: 
\begin{enumerate}
	\item $\NF \vdash A$. 
	\item $A$ is valid in all transitive $\N$-frames. 
	\item $A$ is valid in all finite transitive $\N$-frames. 
	\item $A$ is valid in all finite $\Sub(A)$-transitive $\N$-frames. 
\end{enumerate}
\end{thm}

\begin{thm}[The completeness and finite frame property of $\NRF$]\label{Thm:complNRF}
For any $A \in \MF$, the following are equivalent: 
\begin{enumerate}
	\item $\NRF \vdash A$. 
	\item $A$ is valid in all transitive and serial $\N$-frames. 
	\item $A$ is valid in all finite transitive and serial $\N$-frames. 
	\item $A$ is valid in all finite $\Sub(A)$-transitive and $\Sub(A)$-serial $\N$-frames. 
\end{enumerate}
\end{thm}

\begin{proof}
We prove Fact \ref{Fact1}, Theorems \ref{Thm:complNR}, \ref{Thm:complNF}, and \ref{Thm:complNRF} simultaneously. 
Let $L$ be one of $\NR$, $\NF$, and $\NRF$. 
Assume, however, that the statement (4) is the same as (3) when $L = \N$.

$(1 \Rightarrow 2)$: This is already proved in Corollaries \ref{Cor:NR} and \ref{Cor:NF}. 

$(2 \Rightarrow 3)$: Obvious. 

$(3 \Rightarrow 4)$: Suppose that $A$ is valid in all finite $\N$-frames satisfying the corresponding conditions. 
Let $\mathcal{M} = (W, \{\prec_B\}_{B \in \MF}, \Vdash)$ be any finite $\N$-model whose frame $\mathcal{F} =  (W, \{\prec_B\}_{B \in \MF})$ satisfies the corresponding conditions restricted to $\Sub(A)$. 
For example, if $L = \NF$, then $\mathcal{F}$ is $\Sub(A)$-transitive. 
For each $B \in \MF$, let $\prec^*_B$ be the binary relation on $W$ defined as follows: 
\[
	\prec^*_B : = \begin{cases} \prec_B & \text{if}\ \Box B \in \Sub(A), \\
	\{(x, x) \mid x \in W\} & \text{otherwise}.
	\end{cases}
\]
Let $\mathcal{F}^* := (W, \{\prec^*_B\}_{B \in \MF})$. 

\begin{cl}
If $L \in \{\NR, \NRF\}$, then $\mathcal{F}^*$ is serial. 
\end{cl}
\begin{proof}
Let $x \in W$ and $B \in \MF$. 
\begin{itemize}
	\item If $\Box B \in \Sub(A)$, then there exists a $y \in W$ such that $x \prec_B y$ because $\mathcal{F}$ is $\Sub(A)$-serial. 
	Thus, $x \prec^*_B y$. 
	\item If $\Box B \notin \Sub(A)$, then $x \prec^*_B x$. 
\end{itemize}
We have proved that $\mathcal{F}^*$ is $B$-serial. 
\end{proof}

\begin{cl}
If $L \in \{\NF, \NRF\}$, then $\mathcal{F}^*$ is transitive. 
\end{cl}
\begin{proof}
Let $x, y, z \in W$ and $B \in \MF$ be such that $x \prec^*_{\Box B} y$ and $y \prec^*_B z$. 
\begin{itemize}
	\item If $\Box \Box B \in \Sub(A)$, then $\Box B \in \Sub(A)$, and hence $x \prec_{\Box B} y$ and $y \prec_B z$. 
	Since $\mathcal{F}$ is $\Sub(A)$-transitive, we have $x \prec_B z$. 
	Thus, $x \prec^*_B z$. 
	\item If $\Box \Box B \notin \Sub(A)$, then $x = y$ by the definition of $\prec^*_{\Box B}$. 
		Since $y \prec^*_B z$, we obtain $x \prec^*_B z$.  
\end{itemize}
We have proved that $\mathcal{F}^*$ is $B$-transitive. 
\end{proof}

Therefore, $\mathcal{F}^*$ is a finite $\N$-frame satisfying the corresponding conditions. 
Let $\Vdash^*$ be the satisfaction relation on $\mathcal{F}^*$ defined by $x \Vdash^* p : \iff x \Vdash p$. 
By the supposition, $A$ is valid in $\mathcal{F}^*$. 
In particular, $A$ is valid in $(\mathcal{F}^*, \Vdash^*)$. 
Since $\prec^*_B = \prec_B$ for any $B \in \MF$ with $\Box B \in \Sub(A)$, by Fact \ref{Fact2}, $A$ is also valid in $\mathcal{M}$. 

$(4 \Rightarrow 1)$: We prove the contrapositive. 
Suppose $L \nvdash A$, and we would like to find a corresponding finite $\N$-frame in which $A$ is not valid. 

For each formula $B \in \MF$, let ${\sim}B$ be $C$ if $B$ is of the form $\neg C$ and $\neg B$ otherwise. 
Let $\num{\Sub(A)} : = \Sub(A) \cup \{{\sim}B \mid B \in \Sub(A)\}$. 
We say that $X \subseteq \num{\Sub(A)}$ is \textit{$L$-consistent} if $L \nvdash \neg \bigwedge X$ where $\bigwedge X$ is a conjunction of all elements of $X$. 
Also, $X$ is called \textit{$A$-maximally $L$-consistent} if $X$ is maximal among $L$-consistent subsets of $\num{\Sub(A)}$. 
It is easily shown that every $L$-consistent subset $X$ of $\num{\Sub(A)}$ is extended to an $A$-maximally $L$-consistent set. 

We define the $\N$-model $\mathcal{M} = (W, \{\prec_B\}_{B \in \MF}, \Vdash)$ as follows: 
\begin{itemize}
	\item $W: = \{X \subseteq \num{\Sub(A)} \mid X$ is $A$-maximally $L$-consistent$\}$; 
	\item For $X, Y \in W$, $X \prec_B Y : \iff \Box B \notin X$ or $B \in Y$; 
	\item For each propositional variable $p$ and $X \in W$, $X \Vdash p : \iff p \in X$. 
\end{itemize}

Let $n$ be the number of elements of $\num{\Sub(A)}$. 
We have that the number of elements of $W$ is smaller than $2^n$. 
Since $L \nvdash A$, $\{{\sim}A\}$ is $L$-consistent. 
So, we have $X_A \in W$ such that ${\sim}A \in X_A$. 

\begin{cl}\label{TL}
For any $X \in W$ and $B \in \num{\Sub(A)}$, 
\[
	X \Vdash B \iff B \in X.
\]
\end{cl}
\begin{proof}
We prove the claim by induction on the construction of $B$. 
We only give a proof of the case that $B$ is of the form $\Box C$. 

$(\Rightarrow)$: We prove the contrapositive. 
Suppose $\Box C \notin X$. 
Since $X$ is maximal, $\neg \Box C \in X$. 
Assume, towards a contradiction, that $\{{\sim}C\}$ is $L$-inconsistent. 
Then, $L \vdash C$. 
By \textsc{Nec}, $L \vdash \Box C$. 
This contradicts the $L$-consistency of $X$. 

We proved that $\{{\sim}C\}$ is $L$-consistent. 
Let $Y \in W$ be such that $\{{\sim}C\} \subseteq Y$. 
Since $\Box C \notin X$, we have $X \prec_C Y$ by the definition of $\prec_C$. 
Since ${\sim}C \in Y$, we have $C \notin Y$. 
By the induction hypothesis, $Y \nVdash C$.
We conclude that $X \nVdash \Box C$. 

$(\Leftarrow)$: Suppose $\Box C \in X$. 
Let $Y \in W$ be such that $X \prec_C Y$. 
By the definition of $\prec_C$, we have $C \in Y$. 
By the induction hypothesis, $Y \Vdash C$. 
Hence, $X \Vdash \Box C$. 

\end{proof}

Since $A \notin X_A$, by Claim \ref{TL}, we obtain $X_A \nVdash A$. 
Therefore, $A$ is not valid in $\mathcal{M}$. 
This completes the proof of Fact \ref{Fact1} for $L = \N$.

\begin{cl}
If $L \in \{\NR, \NRF\}$, then $(W, \{\prec_B\}_{B \in \MF})$ is $\Sub(A)$-serial. 
\end{cl}
\begin{proof}
Let $X \in W$ and $\Box B \in \Sub(A)$. 
We distinguish the following two cases: 
\begin{itemize}
	\item Case 1: $\Box B \notin X$. \\
	By the definition of $\prec_B$, we have $X \prec_B X$. 
	
	\item Case 2: $\Box B \in X$. \\
	Suppose, towards a contradiction, that $\{B\}$ is $L$-inconsistent. 
	Then, $L \vdash \neg B$. 
	By the rule \textsc{Ros}, we have $L \vdash \neg \Box B$. 
	This contradicts the $L$-consistency of $X$. 
	Hence, $\{B\}$ is $L$-consistent and there exists a $Y \in W$ such that $B \in Y$. 
	By the definition of $\prec_B$, $X \prec_B Y$. 
\end{itemize}
In either case, we have a $Y \in W$ such that $X \prec_B Y$. 
We conclude that $(W, \{\prec_B\}_{B \in \MF})$ is $\Sub(A)$-serial. 
\end{proof}

\begin{cl}
If $L \in \{\NF, \NRF\}$, then $(W, \{\prec_B\}_{B \in \MF})$ is $\Sub(A)$-transitive. 
\end{cl}
\begin{proof}
Let $X, Y, Z \in W$ and $\Box \Box B \in \Sub(A)$ be such that $X \prec_{\Box B} Y$ and $Y \prec_B Z$. 
If $\Box B \notin X$, then trivially $X \prec_B Z$ by the definition of $\prec_B$. 
If $\Box B \in X$, then $\Box \Box B \in X$ because $L \vdash \Box B \to \Box \Box B$. 
Since $X \prec_{\Box B} Y$, we have $\Box B \in Y$. 
Also, since $Y \prec_B Z$, we have $B \in Z$. 
By the definition of $\prec_B$, we obtain $X \prec_B Z$. 
Therefore, $(W, \{\prec_B\}_{B \in \MF})$ is $\Sub(A)$-transitive. 
\end{proof}

Our proof is finished. 
\end{proof}

Furthermore, from our proofs of Fact \ref{Fact1} and Theorems \ref{Thm:complNR}, \ref{Thm:complNF}, and \ref{Thm:complNRF}, we obtain that the sets of all theorems of $\N$, $\NR$, $\NF$, and $\NRF$ are primitive recursive. 
For example, to show that $A \in \MF$ is $\NF$-unprovable, it is sufficient to find a finite fragment $(W, \{\prec_B\}_{\Box B \in \Sub(A)}, \Vdash)$ of a $\Sub(A)$-transitive $\N$-model in which $A$ is false such that the cardinality of $W$ is smaller than $2^{2n}$ where $n$ is the number of subformulas of $A$. 
Thus, a primitive recursive algorithm that searches for such finite structures determines whether each $A \in \MF$ is provable in $\NF$ or not. 

\section{Arithmetical completeness of $\N$}\label{Sec:N}

It is easy to see that for any provability predicate $\PR_T(x)$ of $T$, $\N \subseteq \PL(\PR_T)$. 
Moreover, by our definition of provability predicates, we have the following theorem: 

\begin{thm}[The arithmetical soundness of $\N$]\label{Thm:AS_N}
For any $A \in \MF$, any provability predicate $\PR_T(x)$ of $T$, and any arithmetical interpretation $f$ based on $\PR_T(x)$, if $\N \vdash A$, then $\PA \vdash f(A)$. 
\end{thm}

In this section, we prove that $\N$ is exactly the provability logic of all provability predicates. 
Moreover, we prove that $\N$ is one of the logics considered in Problem \ref{MProb2}, namely, there exists a $\Sigma_1$ provability predicate $\PR_T(x)$ of $T$ such that $\N = \PL(\PR_T)$.

\begin{thm}[The uniform arithmetical completeness of $\N$]\label{Thm:N}
There exist a $\Sigma_1$ provability predicate $\PR_T(x)$ of $T$ and an arithmetical interpretation $f$ based on $\PR_T(x)$ such that for any $A \in \MF$, $\N \vdash A$ if and only if $T \vdash f(A)$.
\end{thm}

Before proving the theorem, we prepare a primitive recursive function $h$ which plays an important role in our proofs of the theorems in this paper. 
The function $h$ was originally introduced in \cite{Kur20} to prove the existence of a Rosser provability predicate whose provability logic is exactly the logic $\mathsf{KD}$. 

We say that an $\LA$-formula is \textit{propositionally atomic} if it is not a Boolean combination of its proper subformulas. 
For each propositionally atomic formula $\varphi$, we prepare a propositional variable $p_\varphi$. 
We define the primitive recursive mapping $I$ from $\LA$-formulas to propositional formulas as follows: 
\begin{enumerate}
	\item For each propositionally atomic formula $\varphi$, $I(\varphi)$ is $p_\varphi$; 
	\item $I(\neg \varphi)$ is $\neg I(\varphi)$; 
	\item $I(\varphi \circ \psi)$ is $I(\varphi) \circ I(\psi)$ for $\circ \in \{\land, \lor, \to\}$. 
\end{enumerate}
It is clear that $I$ is an injection. 
Let $\varphi$ be an $\LA$-formula and $X$ be a finite set of $\LA$-formulas. 
We say that $\varphi$ is a \textit{tautological consequence} (\textit{t.c.}) of $X$ if $\bigwedge_{\psi \in X} I(\psi) \to I(\varphi)$ is a tautology. 
The method of constructing Rosser provability predicates satisfying $\D{2}$ using truth assignments of classical propositional logic is due to Arai \cite{Ara}, and the idea of using t.c.'s is from \cite{Kur20}.

For each natural number $n$, let $P_{T, n}$ be the set of all $\LA$-formulas having a $T$-proof with the G\"odel number less than or equal to $n$. 
It is proved that the set $\{(n, \varphi) \mid \varphi$ is a t.c.~of $P_{T, n}\}$ is primitive recursive. 
The above notions and sets are formalized in $\PA$. 
In particular, we suppose that $P_{T, n}$ is formalized by using the proof predicate $\Proof_T(x, y)$.

The function $h$ is defined as follows by using the formalized recursion theorem: 

\begin{itemize}
	\item $h(0) = 0$. 
	\item $h(m+1) = \begin{cases} i & \text{if}\ h(m) = 0\\
				& \quad \&\ i = \min \{j \in \omega \setminus \{0\} \mid \neg S(\num{j}) \ \text{is a t.c.~of}\ P_{T, m}\}, \\
			h(m) & \text{otherwise}.
		\end{cases}$
\end{itemize}
Here, $S(x)$ is the $\Sigma_1$ formula $\exists y(h(y) = x)$. 
Unlike the Solovay function by the same name used in the proof of Solovay's arithmetical completeness theorem, our function $h$ does not track the structure of models, but is simply used to refer to the numbers $m$ and $i$ such that $h(m) = 0$ and $h(m + 1) = i \neq 0$. 
For such $m$ and $i$, it is shown that $i \leq m$. 
It follows that $h$ is a primitive recursive function (See \cite[p.~603]{Kur20} for details). 
It is also shown that the following proposition holds. 

\begin{prop}[Cf.~{\cite[Lemma 3.2.]{Kur20}}]\label{Prop:h}
\leavevmode
\begin{enumerate}
	\item $\PA \vdash \forall x \forall y(0 < x < y \land S(x) \to \neg S(y))$. 
	\item $\PA \vdash \neg \Con_T \leftrightarrow \exists x(S(x) \land x \neq 0)$, where $\Con_T$ is the $\Pi_1$ consistency statement $\neg \Prov_T(\gn{0=1})$. 
	\item For each $i \in \omega \setminus \{0\}$, $T \nvdash \neg S(\num{i})$. 
	\item For each $n \in \omega$, $\PA \vdash \forall x \forall y(h(x) = 0 \land h(x+1) = y \land y \neq 0 \to x > \num{n})$. 
\end{enumerate}
\end{prop}

We are ready to prove Theorem \ref{Thm:N}. 

\begin{proof}[Proof of Theorem \ref{Thm:N}]
Let $\langle A_n \rangle_{n \in \omega}$ be a primitive recursive enumeration of all $\N$-unprovable $\LB$-formulas. 
For each $n \in \omega$, let $(W_n, \{\prec_{n, B}\}_{B \in \MF}, \Vdash_n)$ be a primitive recursively constructed finite $\N$-model falsifying $A_n$ (See Fact \ref{Fact2} and the comments in the last paragraph of Section \ref{Sec:Compl}). 
We may assume that $\{W_n\}_{n \in \omega}$ is a pairwise disjoint family of subsets of $\omega$ and $\bigcup_{n \in \omega} W_n = \omega \setminus \{0\}$. 
We may also assume that for each $i > 0$, we can primitive recursively find the unique $n$ such that $i \in W_n$. 
Let $\mathcal{M} = (W, \{\prec_B\}_{B \in \MF}, \Vdash)$ be an $\N$-model defined as follows: 
\begin{itemize}
	\item $W : = \bigcup_{n \in \omega} W_n = \omega \setminus \{0\}$. 
	\item $x \prec_B y : \iff x, y \in W_n$ and $x \prec_{n, B} y$ for some $n \in \omega$. 
	\item $x \Vdash p :\iff x \in W_n$ and $x \Vdash_n p$ for some $n \in \omega$. 
\end{itemize}
We may assume that $\mathcal{M}$ is primitive recursively represented in $\PA$. 
Moreover, we assume that $\PA$ proves basic properties of $\mathcal{M}$. 

For each primitive recursive function $g$ enumerating all theorems of $T$, let $\PR_g(x)$ be the $\Sigma_1$ formula $\exists y(g(y) = x \land \Fml(x))$. 
Then, $\PR_g(x)$ is a provability predicate of $T$. 
We define the arithmetical interpretation $f_g$ based on $\PR_g(x)$ by $f_g(p): \equiv \exists x(S(x) \land x \neq 0 \land x \Vdash p)$. 
From an index of such a function $g$ and $A \in \MF$, the $\LA$-sentence $f_g(A)$ is primitive recursively computed. 
Furthermore, it is shown that each $f_g$ is an injective mapping, and so from an index of $g$ and $f_g(A)$, the $\LB$-formula $A$ is recovered primitive recursively.

Next, we define the primitive recursive function $g_0$ enumerating all theorems of $T$. 
The definition of $g_0$ consists of two procedures. 
The definition starts with Procedure 1. 
The values of $g_0$ are defined step by step in the procedure by referring to $T$-proofs according to the proof predicate $\Proof_T(x, y)$. 
At the first time the value of the function $h$ is non-zero, the definition of $g_0$ switches to Procedure 2. 
By using the formalized recursion theorem, the arithmetical interpretation $f_{g_0}$ based on the provability predicate $\PR_{g_0}(x)$ is used in the definition of $g_0$. 
In the construction, we identify each $\LA$-formula with its G\"odel number. 

\vspace{0.1in}

\textsc{Procedure 1}\\
Stage $m$. 
\begin{itemize}
	\item If $h(m+1) = 0$, then
\[
	g_0(m) = \begin{cases} \varphi & \text{if} \ m\ \text{is a}\ T\text{-proof of}\ \varphi, \\
	0 & \text{otherwise.}
	\end{cases}
\]
Go to Stage $m+1$. 

	\item If $h(m+1) \neq 0$, then go to Procedure 2. 
\end{itemize}

\textsc{Procedure 2}\\
Let $m$, $i \neq 0$ and $n$ be such that $h(m) = 0$, $h(m+1) = i$, and $i \in W_n$. 
Recall from Section \ref{Sec:Pre} that $\{\xi_t\}_{t \in \omega}$ is the primitive recursive enumeration of all $\LA$-formulas arranged in ascending order of G\"odel numbers. 
Define 
\[
	g_0(m+t) = \begin{cases} \xi_t & \text{if}\ \xi_t \equiv f_{g_0}(B)\ \&\ i \Vdash_n \Box B\ \text{for some}\ \Box B \in \Sub(A_n), \\
	0 & \text{otherwise.}
	\end{cases}
\]
The definition of $g_0$ is finished.

\begin{cl}\label{NCL}
$\PA + \Con_T \vdash \forall x \forall y \Bigl(\Fml(x) \to \bigl(\Proof_T(x, y) \leftrightarrow x = g_0(y) \bigr) \Bigr)$. 
\end{cl}
\begin{proof}
We argue in $\PA + \Con_T$: 
By Proposition \ref{Prop:h}.2, $h(x) = 0$ for all $x$. 
Thus, the construction of $g_0$ never switches to Procedure 2. 
Then, for any $\LA$-formula $\varphi$ and number $a$, we have that $a$ is a $T$-proof of $\varphi$ if and only if $\varphi = g_0(a)$. 
\end{proof}

Then, it is shown that for any $\LA$-formula $\varphi$ and $n \in \omega$, $\PA \vdash \Proof_T(\gn{\varphi}, \num{n})$ if and only if $\PA \vdash \gn{\varphi} = g_0(\num{n})$. 
It follows that $\PR_{g_0}(x)$ is a $\Sigma_1$ provability predicate of $T$. 

\begin{cl}\label{Cl:g_0}
Let $i \in W_n$ and $B \in \Sub(A_n)$. 
\begin{enumerate}
	\item If $i \Vdash_n B$, then $\PA \vdash S(\num{i}) \to f_{g_0}(B)$. 
	\item If $i \nVdash_n B$, then $\PA \vdash S(\num{i}) \to \neg f_{g_0}(B)$. 
\end{enumerate}
\end{cl}
\begin{proof}
Clauses 1 and 2 are proved simultaneously by induction on the construction of $B \in \Sub(A_n)$. 
Firstly, we prove the base step of the induction. 
The case that $B$ is $\bot$ is trivial. 
We prove the case that $B$ is a propositional variable $p$. 

1. Suppose $i \Vdash_n p$. 
We have $\PA \vdash S(\num{i}) \to \exists x(S(x) \land x \neq 0 \land x \Vdash p)$, and hence $\PA \vdash S(\num{i}) \to f_{g_0}(p)$. 

2. Suppose $i \nVdash_n p$. 
Since $\PA \vdash S(\num{i}) \to \forall x(S(x) \land x \neq 0 \to x = \num{i})$ by Proposition \ref{Prop:h}.1, we have that $\PA \vdash S(\num{i}) \to \forall x(S(x) \land x \neq 0 \to x \nVdash p)$. 
Equivalently, $\PA \vdash S(\num{i}) \to \neg f_{g_0}(p)$. 

Secondly, we prove the induction step. 
The cases of $\neg$, $\lor$, $\land$, and $\to$ are easy. 
So, we give only a proof of the case that $B$ is of the form $\Box C$, where the claim holds for $C$. 

1. Suppose $i \Vdash_n \Box C$. 
We reason in $\PA + S(\num{i})$: 
Let $m$ be such that $h(m) = 0$ and $h(m+1) = i$. 
Let $t$ be the number such that $\xi_t \equiv f_{g_0}(C)$. 
Since $i \Vdash_n \Box C$ and $\Box C \in \Sub(A_n)$, we have $g_0(m+t) = f_{g_0}(C)$. 
Thus, $\PR_{g_0}(\gn{f_{g_0}(C)})$ holds. 
This means that $f_{g_0}(\Box C)$ holds. 

2. Suppose $i \nVdash_n \Box C$. 
Then, there exists a $j \in W_n$ such that $i \prec_{n, C} j$ and $j \nVdash_n C$. 
By the induction hypothesis, $\PA \vdash S(\num{j}) \to \neg f_{g_0}(C)$. 
Let $p$ be a $T$-proof of $S(\num{j}) \to \neg f_{g_0}(C)$. 

We argue in $\PA + S(\num{i})$: 
Let $m$ be such that $h(m) = 0$ and $h(m+1) = i$. 

If $f_{g_0}(C)$ is output in Procedure 1, then there exists a $T$-proof $q < m$ of $f_{g_0}(C)$. 
It follows that $f_{g_0}(C) \in P_{T, m-1}$. 
Since $m > p$ by Proposition \ref{Prop:h}.4, we have that $S(\num{j}) \to \neg f_{g_0}(C)$ is in $P_{T, m-1}$. 
Hence, $\neg S(\num{j})$ is a t.c.~of $P_{T, m-1}$. 
This contradicts $h(m) = 0$. 

If $f_{g_0}(C)$ is output in Procedure 2, then $f_{g_0}(C) \equiv f_{g_0}(D)$ and $i \Vdash_n \Box D$ for some $\Box D \in \Sub(A_n)$. 
Since $f_{g_0}$ is injective, we have $C \equiv D$. 
So $i \Vdash_n \Box C$, this is a contradiction. 

We have proved that $f_{g_0}(C)$ is not output by $g_0$. 
Thus, $\neg \PR_{g_0}(\gn{f_{g_0}(C)})$ holds, and hence $\neg f_{g_0}(\Box C)$ holds. 
\end{proof}

We finish our proof of Theorem \ref{Thm:N}. 
The implication $\Rightarrow$ is obvious. 
We prove the implication $\Leftarrow$. 
Suppose that $\N \nvdash A$. 
We have that $A \equiv A_n$ for some $n \in \omega$ and $i \nVdash_n A$ for some $i \in W_n$. 
By Claim \ref{Cl:g_0}, $\PA \vdash S(\num{i}) \to \neg f_{g_0}(A)$. 
Since $T \nvdash \neg S(\num{i})$ by Proposition \ref{Prop:h}.3, we obtain $T \nvdash f_{g_0}(A)$. 
\end{proof}

\begin{cor}
\begin{align*}
	\N & = \bigcap \{\PL(\PR_T) \mid \PR_T(x)\ \text{is a provability predicate of}\ T\}, \\
	& = \bigcap \{\PL(\PR_T) \mid \PR_T(x)\ \text{is a}\ \Sigma_1\ \text{provability predicate of}\ T\}. 
\end{align*}
Moreover, there exists a $\Sigma_1$ provability predicate $\PR_T(x)$ of $T$ such that $\N = \PL(\PR_T)$. 
\end{cor}

\section{Arithmetical completeness of $\NF$}\label{Sec:NF}

In this section, we investigate provability predicates satisfying the condition $\D{3}$. 
It is easy to show that for any provability predicate $\PR_T(x)$ satisfying $\D{3}$, $\NF \subseteq \PL(\PR_T)$. 
We prove that $\NF$ is exactly the provability logic of all provability predicates satisfying $\D{3}$. 
Moreover, we prove the following uniform version of arithmetical completeness.

\begin{thm}[The uniform arithmetical completeness of $\NF$]\label{Thm:NF}
There exists a $\Sigma_1$ provability predicate $\PR_T(x)$ of $T$ such that
\begin{enumerate}
	\item for any $A \in \MF$ and any arithmetical interpretation $f$ based on $\PR_T(x)$, if $\NF \vdash A$, then $\PA \vdash f(A)$; and
	\item there exists an arithmetical interpretation $f$ based on $\PR_T(x)$ such that for any $A \in \MF$, $\NF \vdash A$ if and only if $T \vdash f(A)$.
\end{enumerate}
\end{thm}

\begin{proof}
Let $\langle A_n \rangle_{n \in \omega}$ be a primitive recursive enumeration of all $\NF$-unprovable $\LB$-formulas. 
For each $n \in \omega$, let $(W_n, \{\prec_{n, B}\}_{B \in \MF}, \Vdash_n)$ be a primitive recursively constructed finite $\Sub(A_n)$-transitive $\N$-model falsifying $A_n$. 
Let $\mathcal{M}$ be the primitive recursively representable $\N$-model defined as the disjoint union of these finite $\N$-models as in the proof of Theorem \ref{Thm:N}. 

We define the primitive recursive function $g_1$ corresponding to this theorem. 
By the formalized recursion theorem, we use $\PR_{g_1}$ and $f_{g_1}$ in the definition of $g_1$ where $\PR_{g_1}(x)$ is the formula $\exists y(g_1(y) = x \land \Fml(x))$ and $f_{g_1}$ is the arithmetical interpretation based on $\PR_{g_1}(x)$ defined by $f_{g_1}(p) \equiv \exists x(S(x) \land x \neq 0 \land x \Vdash p)$. 
As in the definition of the function $g_0$, the definition of $g_1$ consists of Procedures 1 and 2. 
Moreover, the definition of Procedure 1 is completely same as that of $g_0$, so here we only give the definition of Procedure 2. 
Unlike the function $g_0$, to ensure that $\PR_{g_1}(x)$ satisfies $\D{3}$, in Procedure 2, the function $g_1$ outputs the sentences of the form $\PR_{g_1}(\gn{\varphi})$ for already output formulas $\varphi$.

\vspace{0.1in}

\textsc{Procedure 2}\\
Let $m$, $i \neq 0$, and $n$ be such that $h(m) = 0$, $h(m+1) = i$, and $i \in W_n$. 
Define 
\[
	g_1(m+t) = \begin{cases} \xi_t & \text{if}\ \xi_t \equiv f_{g_1}(B)\ \&\ i \Vdash_n \Box B\ \text{for some}\ \Box B \in \Sub(A_n) \\
											& \quad\ \text{or}\ \xi_t \equiv \PR_{g_1}(\gn{\varphi})\ \&\ g_1(l) = \varphi\ \text{for some}\ \varphi\ \text{and}\ l < m+t, \\
	0 & \text{otherwise.}
	\end{cases}
\]

Since Procedure 1 in the definition of $g_1$ is same as that of $g_0$, the following claim is proved as in the proof of Theorem \ref{Thm:N}. 

\begin{cl}\label{Cl:g_1_eq}
$\PA + \Con_T \vdash \forall x \forall y \Bigl(\Fml(x) \to \bigl(\Proof_T(x, y) \leftrightarrow x = g_1(y) \bigr) \Bigr)$. 
\end{cl}

Hence, $\PR_{g_1}(x)$ is a $\Sigma_1$ provability predicate of $T$. 
We prove that $\PR_{g_1}(x)$ satisfies the condition $\D{3}$. 

\begin{cl}\label{NFCL}
For any $\LA$-formula $\varphi$, $\PA \vdash \PR_{g_1}(\gn{\varphi}) \to \PR_{g_1}(\gn{\PR_{g_1}(\gn{\varphi})})$. 
\end{cl}
\begin{proof}
Since $\PR_{g_1}(\gn{\varphi})$ is a $\Sigma_1$ sentence, $\PA \vdash \PR_{g_1}(\gn{\varphi}) \to \Prov_T(\gn{\PR_{g_1}(\gn{\varphi})})$. 
By Claim \ref{Cl:g_1_eq}, $\PA + \Con_T \vdash \forall x \Bigl(\Fml(x) \to \bigl(\Prov_T(x) \leftrightarrow \PR_{g_1}(x) \bigr) \Bigr)$. 
Thus, we have $\PA + \Con_T \vdash \PR_{g_1}(\gn{\varphi}) \to \PR_{g_1}(\gn{\PR_{g_1}(\gn{\varphi})})$. 

We reason in $\PA + \neg \Con_T + \PR_{g_1}(\gn{\varphi})$: 
By Proposition \ref{Prop:h}.2, there exists an $i \neq 0$ such that $S(i)$ holds. 
Let $m$ and $n$ be such that $h(m) = 0$, $h(m+1) = i$, and $i \in W_n$. 
Since $\PR_{g_1}(\gn{\varphi})$ holds, $\varphi$ is output by $g_1$. 
Let $s$ be such that $\xi_s \equiv \varphi$. 

If $\varphi$ is output in Procedure 1, then $\varphi = g_1(k)$ for some $k < m$.  
If $\varphi$ is output in Procedure 2, then $\varphi = g_1(m + s)$. 
In either case, we have that $\varphi \in \{g_1(0), \ldots, g_1(m+s)\}$. 
Let $u$ be such that $\xi_u \equiv \PR_{g_1}(\gn{\varphi})$. 
Since the G\"odel number of $\PR_{g_1}(\gn{\varphi})$ is larger than that of $\varphi$, we have $s < u$ by the choice of the enumeration $\langle \xi_t \rangle_{t \in \omega}$. 
Since there is an $l < m + u$ such that $g_1(l) = \varphi$, we have that $g_1(m + u) = \PR_{g_1}(\gn{\varphi})$. 
Thus, $\PR_{g_1}(\gn{\PR_{g_1}(\gn{\varphi})})$ holds. 

We have proved $\PA + \neg \Con_T \vdash \PR_{g_1}(\gn{\varphi}) \to \PR_{g_1}(\gn{\PR_{g_1}(\gn{\varphi})})$. 
By the law of excluded middle, we conclude $\PA \vdash \PR_{g_1}(\gn{\varphi}) \to \PR_{g_1}(\gn{\PR_{g_1}(\gn{\varphi})})$. 
\end{proof}

\begin{cl}\label{Cl:g_1}
Let $i \in W_n$ and $B \in \Sub(A_n)$. 
\begin{enumerate}
	\item If $i \Vdash_n B$, then $\PA \vdash S(\num{i}) \to f_{g_1}(B)$. 
	\item If $i \nVdash_n B$, then $\PA \vdash S(\num{i}) \to \neg f_{g_1}(B)$. 
\end{enumerate}
\end{cl}
\begin{proof}
This is proved by induction on the construction of $B \in \Sub(A_n)$. 
We only prove the case that $B$ is of the form $\Box C$. 
Clause 1 is proved in the similar way as in the proof of Claim \ref{Cl:g_0}. 

2. Suppose $i \nVdash_n \Box C$. 
We prove in $\PA + S(\num{i})$ that $f_{g_1}(C)$ is not output by $g_1$. 
We distinguish the following two cases: 

\begin{itemize}
	\item Case 1: $C$ is not of the form $\Box D$. \\
	There exists a $j \in W_n$ such that $i \prec_{n, C} j$ and $j \nVdash_n C$. 
	By the induction hypothesis, $\PA \vdash S(\num{j}) \to \neg f_{g_1}(C)$. 
	Let $p$ be a $T$-proof of $S(\num{j}) \to \neg f_{g_1}(C)$. 

	We proceed in $\PA + S(\num{i})$: 
	Let $m$ be such that $h(m) = 0$ and $h(m+1) = i$. 
	By Proposition \ref{Prop:h}.4, we have $m > p$, and hence $S(\num{j}) \to \neg f_{g_1}(C)$ is in $P_{T, m-1}$. 

	If $f_{g_1}(C)$ is output in Procedure 1, then $f_{g_1}(C) \in P_{T, m-1}$, and hence $\neg S(\num{j})$ is a t.c.~of $P_{T, m-1}$. 
	This contradicts $h(m) = 0$. 

	If $f_{g_1}(C)$ is output in Procedure 2, then $\xi_t \equiv f_{g_1}(C)$ and $g_1(m+t) = f_{g_1}(C)$ for some $t$. 
	Since $C$ is not of the form $\Box D$, there is no $\varphi$ such that $f_{g_1}(C) \equiv \PR_{g_1}(\gn{\varphi})$. 
	By the definition of $g_1$, $f_{g_1}(C) \equiv f_{g_1}(D)$ and $i \Vdash_n \Box D$ for some $\Box D \in \Sub(A_n)$. 
	It follows that $C \equiv D$ and this contradicts $i \nVdash_n \Box C$.  

	\item Case 2: $C$ is of the form $\Box D$. \\
	Then, $\Box \Box D \in \Sub(A_n)$. 
	Since $(W_n, \{\prec_{n, B}\}_{B \in \MF}, \Vdash_n)$ is $\Sub(A_n)$-transitive, $\Box D \to \Box \Box D$ is valid in the model. 
	Since $i \nVdash_n \Box \Box D$, we have $i \nVdash_n \Box D$. 
	By the induction hypothesis, 
	\begin{equation}\label{eq1}
		\PA \vdash S(\num{i}) \to \neg f_{g_1}(\Box D). 
	\end{equation}

	We reason in $\PA + S(\num{i})$: 
	If $f_{g_1}(\Box D)$ is output in Procedure 1, then $f_{g_1}(\Box D) \in P_{T, m-1}$. 
	By (\ref{eq1}) and Proposition \ref{Prop:h}.4, $S(\num{i}) \to \neg f_{g_1}(\Box D)$ is also in $P_{T, m-1}$. 
	Hence, $\neg S(\num{i})$ is a t.c.~of $P_{T, m-1}$, a contradiction. 

	If $f_{g_1}(\Box D)$ is output in Procedure 2, then $\xi_t \equiv f_{g_1}(\Box D)$ and $g_1(m+t) = f_{g_1}(\Box D)$ for some $t$. 
	If $f_{g_1}(\Box D) \equiv f_{g_1}(E)$ and $i \Vdash_n \Box E$ for some $\Box E \in \Sub(A_n)$, then $\Box D \equiv E$ and $i \Vdash_n \Box \Box D$, a contradiction. 
	Thus, we have that $f_{g_1}(\Box D) \equiv \PR_{g_1}(\gn{\varphi})$ and $g_1(l) = \varphi$ for some $\varphi$ and $l < m + t$. 
	Since $g_1(l) = \varphi$, we have that $\PR_{g_1}(\gn{\varphi})$ holds, and hence $f_{g_1}(\Box D)$ holds. 
	This contradicts (\ref{eq1}). 

\end{itemize}
In either case, we have shown in $\PA + S(\num{i})$ that $f_{g_1}(C)$ is not output by $g_1$. 
Therefore, $\neg f_{g_1}(\Box C)$ holds in $\PA + S(\num{i})$. 
\end{proof}

The first clause of the theorem follows from Claims \ref{Cl:g_1_eq} and \ref{NFCL}. 
The second clause follows from Proposition \ref{Prop:h}.3 and Claim \ref{Cl:g_1}. 
\end{proof}

\begin{cor}
\begin{align*}
	\NF & = \bigcap \{\PL(\PR_T) \mid \PR_T(x)\ \text{is a provability predicate of}\ T\ \text{satisfying}\ \D{3}\}, \\
	& = \bigcap \{\PL(\PR_T) \mid \PR_T(x)\ \text{is a}\ \Sigma_1\ \text{provability predicate of}\ T\ \text{satisfying}\ \D{3}\}. 
\end{align*}
Moreover, there exists a $\Sigma_1$ provability predicate $\PR_T(x)$ of $T$ such that $\NF = \PL(\PR_T)$. 
\end{cor}

\section{Arithmetical completeness of $\NR$}\label{Sec:NR}

It is known that for any Rosser provability predicate $\PR_T^{\mathrm{R}}(x)$ of $T$ and any $\LA$-formula $\varphi$, if $T \vdash \neg \varphi$, then $\PA \vdash \neg \PR_T^{\mathrm{R}}(\gn{\varphi})$. 
This fact corresponds to the closure under the rule \textsc{Ros} $\dfrac{\neg A}{\neg \Box A}$, and hence it is shown that $\NR \subseteq \PL(\PR_T^{\mathrm{R}})$. 
Moreover, we obtain the following theorem: 

\begin{thm}[The arithmetical soundness of $\NR$]\label{Thm:AS_NR}
For any $A \in \MF$, any Rosser provability predicate $\PR_T^{\mathrm{R}}(x)$ of $T$, and any arithmetical interpretation $f$ based on $\PR_T^{\mathrm{R}}(x)$, if $\NR \vdash A$, then $\PA \vdash f(A)$.
\end{thm}

Our logic $\NR$ is a candidate for the axiomatization of the logic $L^{\mathrm{R}}$ introduced in Section \ref{Sec:Pre}. 
In this section, we prove that this is the case. 
Namely, we prove that $\NR$ is exactly the provability logic of all Rosser provability predicates. 

Firstly, we prove the coincidence of $\NR$ and $L^{\mathrm{R}}$ without going through arithmetic. 
For each $\mathcal{L}(\Box)$-formula $A$, let $A^{\blacksquare}$ be the $\mathcal{L}(\blacksquare)$-formula obtained from $A$ by replacing every $\Box$ in $A$ with $\blacksquare$. 

\begin{thm}\label{Coincidence}
For any $\mathcal{L}(\Box)$-formula $A$, the following are equivalent: 
\begin{enumerate}
	\item $\NR \vdash A$. 
	\item $\mathsf{GR} \vdash A^\blacksquare$. 
\end{enumerate}
\end{thm}
\begin{proof}
Since the rules $\dfrac{A}{\blacksquare A}$ and $\dfrac{\neg A}{\neg \blacksquare A}$ are admissible in $\GR$, the implication ($1 \Rightarrow 2$) is straightforward. 

$(2 \Rightarrow 1)$: We prove the contrapositive. 
Suppose $\NR \nvdash A$. 
By Theorem \ref{Thm:complNR}, there exists a serial $\N$-model $(W, \{\prec_B\}_{B \in \MF}, \Vdash)$ and an element $w \in W$ such that $w \nVdash A$. 
Let $r$ be any object not in $W$ and define $W^* : = W \cup \{r\}$. 
We define binary relations $\prec_B^*$ on $W^*$ for every $\mathcal{L}(\Box, \blacksquare)$-formula $B$ as follows: 
\[
	\prec_B^* : = \begin{cases} \prec_C & B\ \text{is of the form}\ C^\blacksquare\ \text{for some}\ \mathcal{L}(\Box)\text{-formula}\ C, \\ W^2 & \text{otherwise.} \end{cases} 
\]
We define the satisfaction relation $\Vdash^*$ between the elements of $W^*$ and $\mathcal{L}(\Box, \blacksquare)$-formulas as follows: 
Let $x \in W$, $p$ be any propositional variable, and $B$ be any $\mathcal{L}(\Box, \blacksquare)$-formula, 
\begin{itemize}
	\item $x \Vdash^* p \iff x \Vdash p$; 
	\item $r \Vdash^* p$; 
	\item $x \Vdash^* \Box B$; 
	\item $r \Vdash^* \Box B \iff y \Vdash^* B$ for all $y \in W$; 
	\item $x \Vdash^* \blacksquare B \iff y \Vdash^* B$ for all $y \in W$ such that $x \prec_B^* y$; 
	\item $r \Vdash^* \blacksquare B \iff y \Vdash^* B$ for all $y \in W$; 
	\item $\Vdash^*$ fulfills the usual conditions for $\bot, \neg, \land$, and $\lor$. 
\end{itemize}
By the definition of $\Vdash^*$, it is easy to see that all axioms of $\GR^-$ are true in all $x \in W$. 
Also, it can be shown that all axioms of $\GR^-$ are true in $r$. 
We give only a proof of the fact that $\Box \neg B \to \Box \neg \blacksquare B$ is true in $r$. 

Suppose $r \Vdash^* \Box \neg B$. 
Let $x$ be any element of $W$. 
If $B$ is of the form $C^\blacksquare$ for some $C \in \MF$, then $\prec_B^* = \prec_C$. 
Since $\prec_C$ is serial, we find a $y \in W$ such that $x \prec_C y$. 
Thus, $x \prec_B^* y$. 
Otherwise, we have $\prec_B^* = W^2$, and hence $x \prec_B^* x$. 
In either case, we obtain a $y \in W$ such that $x \prec_B^* y$. 
By the supposition, we have $y \Vdash^* \neg B$, and hence $x \Vdash^* \neg \blacksquare B$. 
We conclude $r \Vdash^* \Box \neg \blacksquare B$. 

It is easy to show that the rules \textsc{MP} and \textsc{Nec} for $\Box$ preserve validity in the model $(W^*, \{\prec_B^*\}, \Vdash^*)$. 
So, we obtain that all theorems of $\GR^-$ are valid in the model. 

By induction on the construction of $\mathcal{L}(\Box)$-formula $B$, we can prove that for any $x \in W$, $x \Vdash B$ if and only if $x \Vdash^* B^\blacksquare$. 
Since $w \nVdash A$, we get $w \nVdash^* A^\blacksquare$. 
Thus, $r \nVdash^* \Box (A^\blacksquare)$. 
Since every theorem of $\GR^-$ is true in all elements of $W^*$, we obtain $\GR^- \nvdash \Box (A^\blacksquare)$. 
By Theorem \ref{GRm}, we conclude that $\GR \nvdash A^\blacksquare$. 
\end{proof}

Therefore, $\NR$ is exactly the $\Box$-free fragment $L^{\mathrm{R}}$ of $\mathsf{GR}$. 
By Corollary \ref{Cor:LR} and Theorem \ref{Coincidence}, we obtain that the provability logic of all Rosser provability predicates of $T$ coincides with $\NR$. 

\begin{cor}\label{Cor:PLRos1}
If $T$ is $\Sigma_1$-sound, then 
\begin{align*}
	\NR & = \bigcap \{\PL(\PR_T^{\mathrm{R}}) \mid \PR_T^{\mathrm{R}}(x)\ \text{is a Rosser provability predicate of}\ T\}.  
\end{align*}
Moreover, there exists a Rosser provability predicate $\PR_T^{\mathrm{R}}(x)$ of $T$ such that $\NR = \PL(\PR_T^{\mathrm{R}})$. 
\end{cor}

Secondly, we directly prove the arithmetical completeness theorem for $\NR$ without using the arithmetical completeness theorem for $\GR$, to show that Corollary \ref{Cor:PLRos1} holds without assuming the $\Sigma_1$-soundness of $T$.

\begin{thm}[The uniform arithmetical completeness of $\NR$]\label{Thm:NR}
There exist a Rosser provability predicate $\PR_T^{\mathrm{R}}(x)$ of $T$ and an arithmetical interpretation $f$ based on $\PR_T^{\mathrm{R}}(x)$ such that for any $A \in \MF$, $\NR \vdash A$ if and only if $T \vdash f(A)$.
\end{thm}
\begin{proof}
Let $\langle A_n \rangle_{n \in \omega}$ be a primitive recursive enumeration of all $\NR$-unprovable $\LB$-formulas. 
For each $n \in \omega$, let $(W_n, \{\prec_{n, B}\}_{B \in \MF}, \Vdash_n)$ be a primitive recursively constructed finite $\Sub(A)$-serial $\N$-model in which $A_n$ is not valid. 
Let $\mathcal{M} = (W, \{\prec_B\}_{B \in \MF}, \Vdash)$ be the $\N$-model defined as in the previous sections. 
We define the corresponding primitive recursive function $g_2$ enumerating all theorems of $T$. 
Let $\PR_{g_2}^{\mathrm{R}}(x)$ be the formula 
\[
	\exists y \bigl(\Fml(x) \land x = g_2(y) \land \forall z < y\, \dot{\neg}(x) \neq g_2(z) \bigr).
\] 
In the definition of $g_2$, we use the arithmetical interpretation $f_{g_2}$ based on $\PR_{g_2}^{\mathrm{R}}(x)$ defined as $f_{g_2}(p) \equiv \exists x (S(x) \land x \neq 0 \land x \Vdash p)$. 
Procedure 1 in the construction of $g_2$ is same as that of $g_0$, and so we only give the definition of Procedure 2. 
Roughly speaking, $g_2$ is defined so that if $S(\num{i})$ holds for $i \in W_n$, then in Procedure 2, for $\Box B \in \Sub(A_n)$ which is false in $i$, $g_2$ outputs $\neg f_{g_2}(B)$ before any output of $f_{g_2}(B)$. 

\vspace{0.1in}

\textsc{Procedure 2}\\
Let $m$, $i \neq 0$, and $n$ be such that $h(m) = 0$, $h(m+1) = i$, and $i \in W_n$. 
We define the finite set $X$ of $\LA$-sentences as follows:
\[
	X : = \{\neg f_{g_2}(B) \mid i \nVdash_n \Box B\ \&\ \Box B \in \Sub(A_n)\}.
\]
Let $\chi_0, \ldots, \chi_{k-1}$ be the listing of all elements of $X$ arranged in descending order of G\"odel numbers. 
For $l < k$, define
\[
	g_2(m + l) = \chi_l. 
\]
And define 
\[
	g_2(m + k + t) = \xi_t. 
\]
The definition of $g_2$ is finished.

\begin{cl}\label{Cl:g_2_eq}\leavevmode
\begin{enumerate}
	\item $\PA + \Con_T \vdash \forall x \forall y \Bigl(\Fml(x) \to \bigl(\Proof_T(x, y) \leftrightarrow x = g_2(y) \bigr) \Bigr)$. 
	\item $\PA \vdash \forall x \Bigl(\Fml(x) \to \bigl(\Prov_T(x) \leftrightarrow \PR_{g_2}(x) \bigr) \Bigr)$. 
\end{enumerate}
\end{cl}
\begin{proof}
Clause 1 is proved similarly as in the proof of Theorem \ref{Thm:N}. 

2. By Clause 1, $\PA + \Con_T \vdash \forall x \Bigl(\Fml(x) \to \bigl(\Prov_T(x) \leftrightarrow \PR_{g_2}(x) \bigr) \Bigr)$. 
Also, $\PA + \neg \Con_T \vdash \forall x \bigl(\Fml(x) \to \Prov_T(x) \bigr)$. 
Proposition \ref{Prop:h}.2 says that $\PA$ verifies that if $T$ is inconsistent, then the construction of $g_2$ eventually switches to Procedure 2. 
Since $g_2$ outputs all $\LA$-formulas in Procedure 2, we have $\PA + \neg \Con_T \vdash \forall x \bigl(\Fml(x) \to \PR_{g_2}(x) \bigr)$. 
Hence, $\PA + \neg \Con_T \vdash \forall x \Bigl(\Fml(x) \to \bigl(\Prov_T(x) \leftrightarrow \PR_{g_2}(x) \bigr) \Bigr)$. 
By the law of excluded middle, we conclude $\PA \vdash \forall x \Bigl(\Fml(x) \to \bigl(\Prov_T(x) \leftrightarrow \PR_{g_2}(x) \bigr) \Bigr)$. 
\end{proof}

It follows from this claim, $\PR_{g_2}^{\mathrm{R}}(x)$ is a Rosser provability predicate of $T$. 

\begin{cl}\label{Cl:g_2}
Let $i \in W_n$ and $B \in \Sub(A_n)$. 
\begin{enumerate}
	\item If $i \Vdash_n B$, then $\PA \vdash S(\num{i}) \to f_{g_2}(B)$. 
	\item If $i \nVdash_n B$, then $\PA \vdash S(\num{i}) \to \neg f_{g_2}(B)$. 
\end{enumerate}
\end{cl}
\begin{proof}
This claim is proved by induction on the construction of $B \in \Sub(A_n)$. 
We only give a proof of the case $B \equiv \Box C$. 

1. Suppose that $i \Vdash_n \Box C$. 
Since $\Box C \in \Sub(A_n)$ and $(W_n, \{\prec_{n, B}\}_{B \in \MF}, \Vdash_n)$ is $\Sub(A_n)$-serial, there is a $j \in W_n$ such that $i \prec_{n, C} j$. 
Then, $j \Vdash_n C$. 
By the induction hypothesis, $\PA \vdash S(\num{j}) \to f_{g_2}(C)$. 
Let $p$ be a $T$-proof of $S(\num{j}) \to f_{g_2}(C)$. 

We argue in $\PA + S(\num{i})$: 
Let $m$ be such that $h(m) = 0$ and $h(m+1) = i$. 
Also, let $X$ be the finite set of $\LA$-formulas as in Procedure 2 and let $k$ be the cardinality of $X$. 
By Proposition \ref{Prop:h}.4, $m > p$, and hence $S(\num{j}) \to f_{g_2}(C)$ is in $P_{T, m-1}$. 

If $\neg f_{g_2}(C) \in P_{T, m-1}$, then $\neg S(\num{j})$ is a t.c.~of $P_{T, m-1}$, a contradiction. 
Hence, $\neg f_{g_2}(C) \notin P_{T, m-1}$, and so $\neg f_{g_2}(C) \notin \{g_2(0), \ldots, g_2(m-1)\}$. 
If $\neg f_{g_2}(C) \in X$, then there exists a $\Box D \in \Sub(A_n)$ such that $\neg f_{g_2}(C) \equiv \neg f_{g_2}(D)$ and $i \nVdash_n \Box D$. 
Then, we have $C \equiv D$, and this contradicts $i \Vdash_n \Box C$. 
Thus, we have $\neg f_{g_2}(C) \notin X$, that is, $\neg f_{g_2}(C) \notin \{g_2(m), \ldots, g_2(m + k -1)\}$. 
Therefore, we obtain $\neg f_{g_2}(C) \notin \{g_2(0), \ldots, g_2(m + k -1)\}$. 

Let $s$ and $u$ be such that $\xi_s \equiv f_{g_2}(C)$ and $\xi_u \equiv \neg f_{g_2}(C)$. 
Then, $s < u$, $g_2(m + k + s) = f_{g_2}(C)$, and $g_2(m + k + u) = \neg f_{g_2}(C)$. 
In particular, $g_2(m + k + u)$ is the first output of $\neg f_{g_2}(C)$ by $g_2$. 
Therefore, $\PR_{g_2}^{\mathrm{R}}(\gn{f_{g_2}(C)})$ holds. 
That is, $f_{g_2}(\Box C)$ holds. 

2. Suppose $i \nVdash_n \Box C$. 
Then, there exists a $j \in W_n$ such that $i \prec_{n, C} j$ and $j \nVdash_n C$. 
By the induction hypothesis, $\PA \vdash S(\num{j}) \to \neg f_{g_2}(C)$. 

We reason in $\PA + S(\num{i})$: 
Let $m$ be such that $h(m) = 0$ and $h(m+1) = i$. 
Also, let $X$ and $k$ be as in the definition of $g_2$. 
As above, $S(\num{j}) \to \neg f_{g_2}(C)$ is in $P_{T, m-1}$. 
If $f_{g_2}(C) \in P_{T, m-1}$, then $\neg S(\num{j})$ is a t.c.~of $P_{T, m-1}$, and this is a contradiction. 
Thus, $f_{g_2}(C) \notin P_{T, m-1}$, and hence $f_{g_2}(C) \notin \{g_2(0), \ldots, g_2(m-1)\}$. 

On the other hand, since $\Box C \in \Sub(A_n)$ and $i \nVdash_n \Box C$, we obtain $\neg f_{g_2}(C) \in X$. 
That is, $\neg f_{g_2}(C) \in \{g_2(m), \ldots, g_2(m+k-1)\}$. 
Since the G\"odel number of $\neg f_{g_2}(C)$ is larger than that of $f_{g_2}(C)$, even if $f_{g_2}(C) \in X$, $\neg f_{g_2}(C)$ is listed earlier than $f_{g_2}(C)$ in the listing $\chi_0, \ldots, \chi_{k-1}$ of $X$.
Thus, $\neg f_{g_2}(C)$ is output by $g_2$ earlier than any output of $f_{g_2}(C)$. 
Therefore, $\neg \PR_{g_2}^{\mathrm{R}}(\gn{f_{g_2}(C)})$ holds. 
This means that $\neg f_{g_2}(\Box C)$ holds. 
\end{proof}

The theorem follows from Proposition \ref{Prop:h}.3 and Claims \ref{Cl:g_2_eq} and \ref{Cl:g_2}. 
\end{proof}

Finally, we obtain that Corollary \ref{Cor:PLRos1} holds regardless of whether $T$ is $\Sigma_1$-sound or not.

\begin{cor}\label{Cor:PLRos2}
\begin{align*}
	\NR & = \bigcap \{\PL(\PR_T^{\mathrm{R}}) \mid \PR_T^{\mathrm{R}}(x)\ \text{is a Rosser provability predicate of}\ T\}.  
\end{align*}
Moreover, there exists a Rosser provability predicate $\PR_T^{\mathrm{R}}(x)$ of $T$ such that $\NR = \PL(\PR_T^{\mathrm{R}})$. 
\end{cor}

\section{Arithmetical completeness of $\NRF$}\label{Sec:NRF}

Arai \cite{Ara} proved the existence of Rosser provability predicates satisfying the condition $\D{3}$. 
For such Rosser provability predicates $\PR_T^{\mathrm{R}}(x)$, one has $\NRF \subseteq \PL(\PR_T^{\mathrm{R}})$. 
In this section, we investigate the provability logic of Arai's predicates, and prove that $\NRF$ is exactly the provability logic of all Rosser provability predicates satisfying $\D{3}$.

\begin{thm}[The uniform arithmetical completeness of $\NRF$]\label{Thm:NRF}
There exists a Rosser provability predicate $\PR_T^{\mathrm{R}}(x)$ of $T$ such that
\begin{enumerate}
	\item for any $A \in \MF$ and any arithmetical interpretation $f$ based on $\PR_T^{\mathrm{R}}(x)$, if $\NRF \vdash A$, then $\PA \vdash f(A)$; and
	\item there exists an arithmetical interpretation $f$ based on $\PR_T^{\mathrm{R}}(x)$ such that for any $A \in \MF$, $\NRF \vdash A$ if and only if $T \vdash f(A)$.
\end{enumerate}
\end{thm}

\begin{proof}
Let $\langle A_n \rangle_{n \in \omega}$ be a primitive recursive enumeration of all $\NRF$-unprovable $\LB$-formulas. 
For each $n \in \omega$, let $(W_n, \{\prec_{n, B}\}_{B \in \MF}, \Vdash_n)$ be a primitive recursively constructed finite $\Sub(A_n)$-transitive and $\Sub(A_n)$-serial $\N$-model falsifying $A_n$. 
Let $\mathcal{M} = (W, \{\prec_B\}_{B \in \MF}, \Vdash)$ be the primitive recursively representable infinite model constructed as the disjoint union of $\{(W_n, \{\prec_{n, B}\}_{B \in \MF}, \Vdash_n)\}_{n \in \omega}$ as in the proof of Theorem \ref{Thm:N}. 
In particular, $W = \omega \setminus \{0\}$. 

Unlike the proofs in the previous sections, our proof of this theorem uses a different function $h'$ instead of the function $h$. 
By using the double recursion theorem, we simultaneously define the primitive recursive functions $h'$ and $g_3$. 
Firstly, we define the function $h'$. 

\begin{itemize}
	\item $h'(0) = 0$. 
	\item $h'(m+1) = \begin{cases}
				i & \text{if}\ h'(m) = 0\\
					& \&\ i = \min \bigl\{j \in \omega \setminus \{0\} \mid \neg S'(\num{j}) \ \text{is a t.c.~of}\ P_{T, m}\\
					& \quad\ \text{or}\ \exists \varphi < m \ \bigl[\neg \varphi \notin P_{T, m} \cup X_{j, m} \\
					& \quad \quad \quad \quad \&\ S'(\num{j}) \to \neg \PR_{g_3}^{\mathrm{R}}(\gn{\varphi})\ \text{is a t.c.~of}\ P_{T, m}\bigr]\bigr\}\\
			h'(m) & \text{otherwise}.
		\end{cases}$
\end{itemize}
Here, $S'(x)$ is the $\Sigma_1$ formula $\exists y(h'(y) = x)$. 
Also, $\varphi < m$ means that the G\"odel number of $\varphi$ is smaller than $m$. 
Also, for each $j \in W_n$ and number $m$, $X_{j, m}$ is the finite set 
\begin{equation*}
	\left\{ \neg f_{g_3}(D) \left|
	\begin{array}{l}
	\Box D \in \Sub(A_n) \\
		\&\ \exists l \in W_n \ \bigl[j \prec_{n, D} l \ \&\ S'(\num{l}) \to \neg f_{g_3}(D)\ \text{is a t.c.~of}\ P_{T, m} \bigr]
	\end{array}
	\right.\right\}, 
\end{equation*}
where $f_{g_3}$ is the arithmetical interpretation based on $\PR_{g_3}^{\mathrm{R}}(x)$ defined as $f_{g_3}(p) : \equiv \exists x (S'(x) \land x \neq 0 \land x \Vdash p)$. 

We find some $r \in \omega$ such that $A_r \equiv \bot$ since $\NRF \nvdash \bot$, and we fix any $j_0 \in W_r$. 
Here, we prove that if $h'(m) = 0$ and $h'(m+1) = i \neq 0$, then $i \leq \max\{j_0, m\}$. 
It follows that $h'$ is actually a primitive recursive function. 
Suppose $h'(m) = 0$ and $h'(m+1) = i \neq 0$. 
If $P_{T, m}$ is propositionally unsatisfiable, then $\neg S'(\num{1})$ is a t.c.~of $P_{T, m}$, and then $i = 1 \leq \max\{j_0, m\}$. 
So, we may assume that $P_{T, m}$ is propositionally satisfiable. 
\begin{itemize}
	\item Suppose that $\neg S'(\num{i})$ is a t.c.~of $P_{T, m}$, then $S'(\num{i})$ is a subformula of a formula contained in $P_{T, m}$ because $S'(\num{i})$ is propositionally atomic. 
Then, the G\"odel number of $S'(\num{i})$ is smaller than $m$, and hence $i \leq m \leq \max\{j_0, m\}$. 
	\item Suppose that there exists a sentence $\varphi$ such that $\neg \varphi \notin P_{T, m} \cup X_{i, m}$ and $S'(\num{i}) \to \neg \PR_{g_3}^{\mathrm{R}}(\gn{\varphi})$ is a t.c.~of $P_{T, m}$. 
	\begin{itemize}
		\item If $\neg \PR_{g_3}^{\mathrm{R}}(\gn{\varphi})$ is a t.c.~of $P_{T, m}$, then so is $S'(\num{j_0}) \to \neg \PR_{g_3}^{\mathrm{R}}(\gn{\varphi})$. 
		Since $A_r$ has no subformulas of the form $\Box D$, we have that $X_{j_0, m} = \emptyset$. 
		So, $\neg \varphi \notin P_{T, m} \cup X_{j_0, m}$. 
		We obtain $i \leq j_0 \leq \max\{j_0, m\}$. 
		
		\item If $\neg \PR_{g_3}^{\mathrm{R}}(\gn{\varphi})$ is not a t.c.~of $P_{T, m}$, then $\neg S'(\num{i})$ is a t.c.~of the propositionally satisfiable set $P_{T,m} \cup \{\PR_{g_3}^{\mathrm{R}}(\gn{\varphi})\}$.  
		Since $S'(\num{i})$ is distinct from propositionally atomic $\PR_{g_3}^{\mathrm{R}}(\gn{\varphi})$, we have that $S'(\num{i})$ is a subformula of a formula in $P_{T,m}$ as above. 
		We have $i \leq m \leq \max\{j_0, m\}$. 
	\end{itemize}
\end{itemize}
We have shown $i \leq \max\{j_0, m\}$. 
Moreover, the above argument can be carried out in $\PA$.

Secondly, we define the function $g_3$. 
We only give the definition of Procedure 2 of the construction of $g_3$. 

\vspace{0.1in}

\textsc{Procedure 2}\\
Let $m$, $i \neq 0$, and $n$ be such that $h'(m) = 0$, $h'(m+1) = i$, and $i \in W_n$. 
Let $\chi_0, \ldots, \chi_{k-1}$ be the listing of all elements of the set $X_{i, m-1}$ arranged in descending order of G\"odel numbers. 
For $l < k$, define
\[
	g_3(m + l) = \chi_l. 
\]
And define 
\[
	g_3(m + k + t) = \xi_t. 
\]
The definition of $g_3$ is finished. 
The construction of $g_3$ is exactly the same as that of $g_2$ in the proof of Theorem \ref{Thm:NR}, except that it is based on the family of $\N$-models corresponding to the logic $\NRF$ and uses $X_{i, m-1}$ and $h'$ instead of $X$ and $h$, respectively.

Similarly to Proposition \ref{Prop:h}, the following claim holds. 

\begin{cl}\label{Cl:h}
\leavevmode
\begin{enumerate}
	\item $\PA \vdash \forall x \forall y(0 < x < y \land S'(x) \to \neg S'(y))$. 
	\item $\PA \vdash \neg \Con_T \leftrightarrow \exists x(S'(x) \land x \neq 0)$. 
	\item For each $i \in \omega \setminus \{0\}$, $T \nvdash \neg S'(\num{i})$. 
	\item For each $n \in \omega$, $\PA \vdash \forall x \forall y(h'(x) = 0 \land h'(x + 1) = y \land y \neq 0 \to x > \num{n})$. 
\end{enumerate}
\end{cl}
\begin{proof}
1. This is straightforward from the definition of $h'$. 

2. The implication $\to$ is easy, and so we prove the implication $\leftarrow$. 
Argue in $\PA$: 
Suppose that $S'(i)$ holds for some $i \neq 0$. 
Let $m$ and $n$ be such that $h'(m) = 0$, $h'(m+1) = i$, and $i \in W_n$. 
Also, let $k$ be the cardinality of the set $X_{i, m-1}$. 
We would like to show that $T$ is inconsistent. 
We distinguish the following two cases: 
\begin{itemize}
	\item Case 1: $\neg S'(\num{i})$ is a t.c.~of $P_{T, m}$. \\
	Then, $\neg S'(\num{i})$ is $T$-provable. 
	Since $S'(\num{i})$ is a true $\Sigma_1$ sentence, it is provable in $T$. 
	Therefore, $T$ is inconsistent. 

	\item Case 2: There exists a $\varphi$ such that $\neg \varphi \notin P_{T, m} \cup X_{i, m}$ and $S'(\num{i}) \to \neg \PR_{g_3}^{\mathrm{R}}(\gn{\varphi})$ is a t.c.~of $P_{T, m}$. \\
	Then, $\neg \varphi \notin P_{T, m-1} \cup X_{i, m-1}$. 
	Hence, $\neg \varphi \notin \{g_3(0), \ldots, g_3(m + k -1)\}$. 
	Let $s$ and $u$ be such that $\xi_s \equiv \varphi$ and $\xi_u \equiv \neg \varphi$. 
	We have that $s < u$, $g_3(m + k + s) = \varphi$, and $g_3(m + k + u) = \neg \varphi$. 
	In particular, $g_3(m + k + u)$ is the first output of $\neg \varphi$ by $g_3$. 
	Hence, $\PR_{g_3}^{\mathrm{R}}(\gn{\varphi})$ holds. 
	Then, $S'(\num{i}) \land \PR_{g_3}^{\mathrm{R}}(\gn{\varphi})$ is a true $\Sigma_1$ sentence, and so it is provable in $T$.
	Since $S'(\num{i}) \to \neg \PR_{g_3}^{\mathrm{R}}(\gn{\varphi})$ is also provable in $T$, we have that $T$ is inconsistent. 
\end{itemize}

3. Suppose $T \vdash \neg S'(\num{i})$ for $i \neq 0$. 
Let $p$ be a $T$-proof of $\neg S'(\num{i})$. 
Then, $\neg S'(\num{i}) \in P_{T, p}$, and thus $h'(p+1) \neq 0$. 
This means that $\exists x(S'(x) \land x \neq 0)$ is true. 
By clause 2, $T$ is inconsistent, a contradiction. 

4. Since $T$ is consistent, we have that $h'(m+1) = 0$ for all $m \in \omega$ by clause 2. 
So, clause 4 is immediately obtained. 
\end{proof}

The following claim is proved as in the proof of Theorem \ref{Thm:NR}.

\begin{cl}\label{Cl:g_3_eq}\leavevmode
\begin{enumerate}
	\item $\PA + \Con_T \vdash \forall x \forall y \Bigl(\Fml(x) \to \bigl(\Proof_T(x, y) \leftrightarrow x = g_3(y) \bigr) \Bigr)$. 
	\item $\PA \vdash \forall x \Bigl(\Fml(x) \to \bigl(\Prov_T(x) \leftrightarrow \PR_{g_3}(x) \bigr) \Bigr)$. 
\end{enumerate}
\end{cl}

Hence, $\PR_{g_3}^{\mathrm{R}}(x)$ is a Rosser provability predicate of $T$. 

\begin{cl}\label{Cl:g_3}
Let $i \in W_n$ and $B \in \Sub(A_n)$. 
\begin{enumerate}
	\item If $i \Vdash_n B$, then $\PA \vdash S'(\num{i}) \to f_{g_3}(B)$. 
	\item If $i \nVdash_n B$, then $\PA \vdash S'(\num{i}) \to \neg f_{g_3}(B)$. 
\end{enumerate}
\end{cl}
\begin{proof}
We prove the claim by induction on the construction of $B \in \Sub(A_n)$. 
We only give a proof of the case $B \equiv \Box C$. 

1. Suppose that $i \Vdash_n \Box C$. 
For each $l \in W_n$ with $l \Vdash_n C$, by the induction hypothesis, we have $\PA \vdash S'(\num{l}) \to f_{g_3}(C)$. 
Since $W_n$ is finite and the model $(W_n, \{\prec_{n, B}\}_{B \in \MF}, \Vdash_n)$ is primitive recursively represented, we find a $p \in \omega$ such that
\begin{equation}\label{fml1}
	\PA \vdash \forall x \in W_n \bigl(x \Vdash_n C \to \exists y < \num{p}\, \Proof_T(\gn{S'(\dot{x}) \to f_{g_3}(C)}, y) \bigr).
\end{equation}
Since $\Box C \in \Sub(A_n)$ and $(W_n, \{\prec_{n, B}\}_{B \in \MF}, \Vdash_n)$ is $\Sub(A_n)$-serial, there exists a $j \in W_n$ such that $i \prec_{n, C} j$. 
Then, $j \Vdash_n C$ and there exists a $T$-proof of $S'(\num{j}) \to f_{g_3}(C)$ smaller than $p$. 

We argue in $\PA + S'(\num{i})$: 
Let $m$ be such that $h'(m) = 0$ and $h'(m+1) = i$. 

If $\neg f_{g_3}(C) \in P_{T, m-1}$, then $\neg S'(\num{j})$ is a t.c.~of $P_{T, m-1}$ because $m > p$ by Claim \ref{Cl:h}.4. 
We have $h'(m) \neq 0$ by the definition of $h'$, and this is a contradiction. 
Hence, $\neg f_{g_3}(C) \notin P_{T, m-1}$. 
If $\neg f_{g_3}(C) \in X_{i, m-1}$, then there exist $\Box D \in \Sub(A_n)$ and $l \in W_n$ such that $\neg f_{g_3}(C) \equiv \neg f_{g_3}(D)$, $i \prec_{n, D} l$, and $S'(\num{l}) \to \neg f_{g_3}(D)$ is a t.c.~of $P_{T, m-1}$.  
We have $C \equiv D$. 
Since $i \Vdash_n \Box C$ and $i \prec_{n, C} l$, we have $l \Vdash_n C$. 
By (\ref{fml1}), we obtain that $S'(\num{l}) \to f_{g_3}(C)$ has a $T$-proof smaller than $p$. 
Since $m > p$ by Claim \ref{Cl:h}.4, $S'(\num{l}) \to f_{g_3}(C)$ is in $P_{T, m-1}$. 
Since both $S'(\num{l}) \to f_{g_3}(C)$ and $S'(\num{l}) \to \neg f_{g_3}(C)$ are t.c.'s of $P_{T, m-1}$, we have that $\neg S'(\num{l})$ is a t.c.~of $P_{T, m-1}$, a contradiction with $h'(m) = 0$. 
Thus, we have $\neg f_{g_3}(C) \notin X_{i, m-1}$, that is, $\neg f_{g_3}(C) \notin \{g_3(m), \ldots, g_3(m + k -1)\}$. 
Therefore, we obtain $\neg f_{g_3}(C) \notin \{g_3(0), \ldots, g_3(m + k -1)\}$. 

Let $s$ and $u$ be such that $\xi_s \equiv f_{g_3}(C)$ and $\xi_u \equiv \neg f_{g_3}(C)$. 
Then, $s < u$, $g_3(m + k + s) = f_{g_3}(C)$, $g_3(m + k + u) = \neg f_{g_3}(C)$, and this is the first $g_3$-output of $\neg f_{g_3}(C)$. 
Therefore, $\PR_{g_3}^{\mathrm{R}}(\gn{f_{g_3}(C)})$ holds. 
That is, $f_{g_3}(\Box C)$ holds. 

2. Suppose $i \nVdash_n \Box C$. 
There exists a $j \in W_n$ such that $i \prec_{n, C} j$ and $j \nVdash_n C$. 
By the induction hypothesis, $\PA \vdash S'(\num{j}) \to \neg f_{g_3}(C)$. 
Let $p$ be a $T$-proof of $S'(\num{j}) \to \neg f_{g_3}(C)$. 

We reason in $\PA + S'(\num{i})$: 
Let $m$ be such that $h'(m) = 0$ and $h'(m+1) = i$. 
Since $m > p$ by Claim \ref{Cl:h}.4, $S'(\num{j}) \to \neg f_{g_3}(C)$ is a t.c.~of $P_{T, m-1}$, and hence we have $\neg f_{g_3}(C) \in X_{i, m-1}$. 
If $f_{g_3}(C) \in P_{T, m-1}$, then $\neg S'(\num{j})$ is a t.c.~of $P_{T, m-1}$, a contradiction. 
Hence, $f_{g_3}(C) \notin \{g_3(0), \ldots, g_3(m-1)\}$. 
Then, even if $f_{g_3}(C) \in X_{i, m-1}$, we get that $g_3$ outputs $\neg f_{g_3}(C)$ earlier than any output of $f_{g_3}(C)$. 
Hence, $\neg \PR_{g_3}^{\mathrm{R}}(\gn{f_{g_3}(C)})$ holds. 
That is, $\neg f_{g_3}(\Box C)$ holds. 
\end{proof}

We prove that $\PR_{g_3}^{\mathrm{R}}(x)$ satisfies $\D{3}$.

\begin{cl}\label{Cl:g_3_4}
For any $\LA$-formula $\varphi$, $\PA \vdash \PR_{g_3}^{\mathrm{R}}(\gn{\varphi}) \to \PR_{g_3}^{\mathrm{R}}(\gn{\PR_{g_3}^{\mathrm{R}}(\gn{\varphi})})$. 
\end{cl}
\begin{proof}
Since $\PR_{g_3}^{\mathrm{R}}(\gn{\varphi})$ is a $\Sigma_1$ sentence, $\PA \vdash \PR_{g_3}^{\mathrm{R}}(\gn{\varphi}) \to \Prov_T(\gn{\PR_{g_3}^{\mathrm{R}}(\gn{\varphi})})$. 
By Claim \ref{Cl:g_3_eq}.1, we have 
\[
	\PA + \Con_T \vdash \forall x \Bigl(\Fml(x) \to \bigl(\Prov_T^{\mathrm{R}}(x) \leftrightarrow \PR_{g_3}^{\mathrm{R}}(x) \bigr) \Bigr).
\]
Since $\PA + \Con_T \vdash \forall x \Bigl(\Fml(x) \to \bigl(\Prov_T(x) \leftrightarrow \Prov_T^{\mathrm{R}}(x) \bigr) \Bigr)$, we obtain $\PA + \Con_T \vdash \forall x \Bigl(\Fml(x) \to \bigl(\Prov_T(x) \leftrightarrow \PR_{g_3}^{\mathrm{R}}(x) \bigr) \Bigr)$.
Thus, 
\[
	\PA + \Con_T \vdash \PR_{g_3}^{\mathrm{R}}(\gn{\varphi}) \to \PR_{g_3}^{\mathrm{R}}(\gn{\PR_{g_3}^{\mathrm{R}}(\gn{\varphi})}).
\] 

We reason in $\PA + \neg \Con_T + \neg \PR_{g_3}^{\mathrm{R}}(\gn{\PR_{g_3}^{\mathrm{R}}(\gn{\varphi})})$: 
By Claim \ref{Cl:h}.2, there exists an $i \neq 0$ such that $S'(i)$ holds. 
Let $m$ and $n$ be such that $h'(m) = 0$, $h'(m+1) = i$, and $i \in W_n$. 
If $\neg \varphi \in P_{T, m-1}$, then $\varphi \notin P_{T, m-1}$ because $\neg S(\num{j})$ is not a t.c.~of $P_{T, m-1}$ for all $j \neq 0$. 
In this case, $\neg \PR_{g_3}^{\mathrm{R}}(\gn{\varphi})$ holds. 

Therefore, in the following, we assume that $\neg \varphi \notin P_{T, m-1}$. 
Let $k$ be the cardinality of the set $X_{i, m-1}$. 
Since $\neg \PR_{g_3}^{\mathrm{R}}(\gn{\PR_{g_3}^{\mathrm{R}}(\gn{\varphi})})$ holds, $\neg \PR_{g_3}^{\mathrm{R}}(\gn{\varphi})$ is output by $g_3$ earlier than any output of $\PR_{g_3}^{\mathrm{R}}(\gn{\varphi})$. 
Let $s$ and $u$ be such that $\xi_s \equiv \PR_{g_3}^{\mathrm{R}}(\gn{\varphi})$ and $\xi_u \equiv \neg \PR_{g_3}^{\mathrm{R}}(\gn{\varphi})$. 
Then, $g_3(m + k + u) = \neg \PR_{g_3}^{\mathrm{R}}(\gn{\varphi})$. 
Since $s < u$ and $g_3(m + k + s) = \PR_{g_3}^{\mathrm{R}}(\gn{\varphi})$, $g_3(m + k + u)$ is not the first output of $\neg \PR_{g_3}^{\mathrm{R}}(\gn{\varphi})$. 
It follows that $\neg \PR_{g_3}^{\mathrm{R}}(\gn{\varphi}) \in P_{T, m-1} \cup X_{i, m-1}$. 
We would like to show that $\neg \PR_{g_3}^{\mathrm{R}}(\gn{\varphi})$ holds. 
We distinguish the following two cases: 

\begin{itemize}
	\item Case 1: $\neg \PR_{g_3}^{\mathrm{R}}(\gn{\varphi}) \in P_{T, m-1}$. \\
	If $\neg \varphi \notin P_{T, m-1} \cup X_{i, m-1}$, then $0 \neq h'(m) \leq i$ because the G\"odel number of $\varphi$ is smaller than $m-1$ by Claim \ref{Cl:h} and $S'(\num{i}) \to \neg \PR_{g_3}^{\mathrm{R}}(\gn{\varphi})$ is a t.c.~of $P_{T, m-1}$.  
	This is a contradiction. 
	Hence, $\neg \varphi \in P_{T, m-1} \cup X_{i, m-1}$. 
	Since $\neg \varphi \notin P_{T, m-1}$ by the assumption, we have $\neg \varphi \in X_{i, m-1}$. 
	Then, there exist $\Box D \in \Sub(A_n)$ and $l \in W_n$ such that $\neg \varphi \equiv \neg f_{g_3}(D)$, $i \prec_{n, D} l$, and $S'(\num{l}) \to \neg f_{g_3}(D)$ is a t.c.~of $P_{T, m-1}$. 
	Since $S'(\num{l}) \to \neg \varphi$ is a t.c.~of $P_{T, m-1}$ but $\neg S'(\num{l})$ is not, we have $\varphi \notin P_{T, m-1}$. 
	Thus, $\varphi \notin \{g_3(0), \ldots, g_3(m-1)\}$. 
	Then, even if $\varphi \in X_{i, m-1}$, $\neg \varphi$ is output by $g_3$ earlier than any output of $\varphi$. 
	Therefore, $\neg \PR_{g_3}^{\mathrm{R}}(\gn{\varphi})$ holds. 

	\item Case 2: $\neg \PR_{g_3}^{\mathrm{R}}(\gn{\varphi}) \in X_{i, m-1}$. \\
	Then, there exist $\Box D \in \Sub(A_n)$ and $j \in W_n$ such that $\neg \PR_{g_3}^{\mathrm{R}}(\gn{\varphi}) \equiv \neg f_{g_3}(D)$, $i \prec_{n, D} j$, and $S'(\num{j}) \to \neg f_{g_3}(D)$ is a t.c.~of $P_{T, m-1}$. 
	It follows that $\PR_{g_3}^{\mathrm{R}}(\gn{\varphi}) \equiv f_{g_3}(D)$. 
	By the definition of $f_{g_3}$, there exists a $\Box E \in \Sub(A_n)$ such that $D \equiv \Box E$ and $\PR_{g_3}^{\mathrm{R}}(\gn{\varphi}) \equiv \PR_{g_3}^{\mathrm{R}}(\gn{f_{g_3}(E)})$. 
	We have that $\varphi \equiv f_{g_3}(E)$, $\Box \Box E \in \Sub(A_n)$, and $i \prec_{n, \Box E} j$. 

	If $\neg f_{g_3}(E) \notin X_{j, m-1}$, then $\neg f_{g_3}(E) \notin P_{T, m-1} \cup X_{j, m-1}$ by the assumption. 
	Since the G\"odel number of $f_{g_3}(E)$ is smaller than $m-1$ and $S'(\num{j}) \to \neg \PR_{g_3}^{\mathrm{R}}(\gn{f_{g_3}(E)})$ is a t.c.~of $P_{T, m-1}$, we have $h'(m) \neq 0$. 
	This is a contradiction. 
	Hence, $\neg f_{g_3}(E) \in X_{j, m-1}$. 
	Then, there exists an $l \in W_n$ such that $j \prec_{n, E} l$ and $S'(\num{l}) \to \neg f_{g_3}(E)$ is a t.c.~of $P_{T, m-1}$. 
	Since $i \prec_{n, \Box E} j$ and $j \prec_{n, E} l$, we obtain $i \prec_{n, E} l$ because $(W_n, \{\prec_{n, B}\}_{B \in \MF}, \Vdash_n)$ is $\Sub(A_n)$-transitive. 
	Therefore, $\neg f_{g_3}(E) \in X_{i, m-1}$, and hence $\neg \varphi \in X_{i, m-1}$. 

	Since $S'(\num{l}) \to \neg \varphi$ is a t.c.~of $P_{T, m-1}$, we have $\varphi \notin P_{T, m-1}$. 
	Hence, $\varphi \notin \{g_3(0), \ldots, g_3(m-1)\}$. 
	Then, even if $\varphi \in X_{i, m-1}$, $g_3$ outputs $\neg \varphi$ earlier than any output of $\varphi$. 
	Therefore, $\neg \PR_{g_3}^{\mathrm{R}}(\gn{\varphi})$ holds. 
\end{itemize}

	We have proved $\PA + \neg \Con_T \vdash \PR_{g_3}^{\mathrm{R}}(\gn{\varphi}) \to \PR_{g_3}^{\mathrm{R}}(\gn{\PR_{g_3}^{\mathrm{R}}(\gn{\varphi})})$. 
By the law of excluded middle, $\PA \vdash \PR_{g_3}^{\mathrm{R}}(\gn{\varphi}) \to \PR_{g_3}^{\mathrm{R}}(\gn{\PR_{g_3}^{\mathrm{R}}(\gn{\varphi})})$. 
\end{proof}

The first clause of the theorem follows from Claims \ref{Cl:g_3_eq} and \ref{Cl:g_3_4}. 
The second clause follows from Claims \ref{Cl:h}.3 and \ref{Cl:g_3}. 
\end{proof}

\begin{cor}
\begin{align*}
	\NRF & = \bigcap \{\PL(\PR_T^{\mathrm{R}}) \mid \PR_T^{\mathrm{R}}(x)\ \text{is a Rosser provability predicate of}\ T\ \text{satisfying}\ \D{3}\}.  
\end{align*}
Moreover, there exists a Rosser provability predicate $\PR_T^{\mathrm{R}}(x)$ of $T$ such that $\NRF = \PL(\PR_T^{\mathrm{R}})$. 
\end{cor}

\section*{Acknowledgement}

This work was supported by JSPS KAKENHI Grant Number JP19K14586. 
The author would like to thank Sohei Iwata, Haruka Kogure, and Yuya Okawa for their helpful comments. 
The author would also like to thank the anonymous referees for their insightful suggestions.

\bibliographystyle{plain}
\bibliography{ref}

\appendix

\section{Appendix: $\Sigma_1$ provability predicates corresponding to $\K$}\label{Sec:K}

In \cite{Kur18_1}, it is proved that there exists a $\Sigma_2$ provability predicate $\PR_T(x)$ of $T$ such that $\K = \PL(\PR_T)$. 
It follows
\[
	\K = \bigcap \{\PL(\PR_T) \mid \PR_T(x)\ \text{is a provability predicate of}\ T\ \text{satisfying}\ \D{2}\}. 
\]
As in Theorems \ref{Thm:N} and \ref{Thm:NF}, we prove that the provability logic of all $\Sigma_1$ provability predicates satisfying $\D{2}$ is also $\K$.

\begin{thm}[The uniform arithmetical completeness of $\K$]\label{Thm:K}
There exists a $\Sigma_1$ provability predicate $\PR_T(x)$ of $T$ such that
\begin{enumerate}
	\item for any $A \in \MF$ and any arithmetical interpretation $f$ based on $\PR_T(x)$, if $\K \vdash A$, then $\PA \vdash f(A)$; and
	\item there exists an arithmetical interpretation $f$ based on $\PR_T(x)$ such that for any $A \in \MF$, $\K \vdash A$ if and only if $T \vdash f(A)$.
\end{enumerate}
\end{thm}
\begin{proof}
Let $(W, \prec, \Vdash)$ be a primitive recursively representable Kripke model satisfying the following conditions: 
\begin{itemize}
	\item $W = \omega \setminus \{0\}$, 
	\item $(W, \prec, \Vdash)$ is the disjoint union of finite Kripke models and for every $i \in W$, we can primitive recursively find the finite set $\{j \in W \mid i \prec j\}$ that may be empty, 
	\item for any $\K$-unprovable $\LB$-formula $A$, there exists an $i \in W$ such that $i \nVdash A$. 
\end{itemize}

We define the primitive recursive function $g_4$ corresponding to this theorem. 
We only describe Procedure 2. 

\vspace{0.1in}

\textsc{Procedure 2}\\
Let $m$ and $i \neq 0$ be such that $h(m) = 0$ and $h(m+1) = i$. 
Define
\[
	g_4(m + t) = \begin{cases}
	\xi_t & \text{if}\ \xi_t\ \text{is a t.c.~of}\ P_{T, m-1} \cup \bigl\{\bigvee_{i \prec j} S(\num{j}) \bigr\}, \\
	0 & \text{otherwise}.
	\end{cases}
\]
Notice that the empty disjunction represents $0=1$. 
Our definition of $g_4$ is finished. 
Let $f_{g_4}$ be the arithmetical interpretation based on $\PR_{g_4}(x)$ defined by $f_{g_4}(p) \equiv \exists x(S(x) \land x \neq 0 \land x \Vdash p)$. 

The following claim is proved similarly as in the proof of Theorem \ref{Thm:N}. 

\begin{cl}\label{Cl:g_4_eq}
$\PA + \Con_T \vdash \forall x \forall y \Bigl(\Fml(x) \to \bigl(\Proof_T(x, y) \leftrightarrow x = g_4(y) \bigr) \Bigr)$. 
\end{cl}

Thus, $\PR_{g_4}(x)$ is a $\Sigma_1$ provability predicate of $T$. 

\begin{cl}\label{Cl:g_4_eq2}
$\PA$ proves the following statement: 
``Let $m$ and $i \neq 0$ be such that $h(m) = 0$ and $h(m+1) = i$. 
Then, for any $\LA$-formula $\varphi$, 
\[
	\PR_{g_4}(\gn{\varphi})\ \text{holds}\ \iff \varphi\ \text{is a t.c.~of}\ P_{T, m-1} \cup \bigl\{\bigvee_{i \prec j} S(\num{j}) \bigr\}\text{''}. 
\]
\end{cl}
\begin{proof}
$(\Rightarrow)$: This is because if $\xi_t \in P_{T, m-1}$, then $\xi_t$ is a t.c.~of $P_{T, m-1} \cup \bigl\{\bigvee_{i \prec j} S(\num{j}) \bigr\}$. 

$(\Leftarrow)$: Immediate from the definition of $g_4$. 
\end{proof}

\begin{cl}\label{Cl:g_4_2}
$\PA \vdash \forall x \forall y \bigl(\PR_{g_4}(x \dot{\to} y) \land \PR_{g_4}(x) \to \PR_{g_4}(y) \bigr)$. 
\end{cl}
\begin{proof}
Since $\PA \vdash \forall x \forall y \bigl(\Prov_T(x \dot{\to} y) \land \Prov_T(x) \to \Prov_T(y) \bigr)$, we have $\PA + \Con_T \vdash \forall x \forall y \bigl(\PR_{g_4}(x \dot{\to} y) \land \PR_{g_4}(x) \to \PR_{g_4}(y) \bigr)$ by Claim \ref{Cl:g_4_eq}. 

We argue in $\PA + \neg \Con_T$: 
By Proposition \ref{Prop:h}.2, there exists an $i \neq 0$ such that $S(\num{i})$ holds. 
Let $m$ be such that $h(m) = 0$ and $h(m+1) = i$. 
Suppose $\PR_{g_4}(\gn{\varphi \to \psi})$ and $\PR_{g_4}(\gn{\varphi})$ hold. 
By Claim \ref{Cl:g_4_eq2}, both $\varphi \to \psi$ and $\varphi$ are t.c.'s of $P_{T, m-1} \cup \bigl\{\bigvee_{i \prec j} S(\num{j}) \bigr\}$. 
We have that $\psi$ is also a t.c.~of $P_{T, m-1} \cup \bigl\{\bigvee_{i \prec j} S(\num{j}) \bigr\}$. 
By Claim \ref{Cl:g_4_eq2} again, we obtain that $\PR_{g_4}(\gn{\psi})$ holds. 

We have proved $\PA + \neg \Con_T \vdash \forall x \forall y \bigl(\PR_{g_4}(x \dot{\to} y) \land \PR_{g_4}(x) \to \PR_{g_4}(y) \bigr)$. 
By the law of excluded middle, we conclude $\PA \vdash \forall x \forall y \bigl(\PR_{g_4}(x \dot{\to} y) \land \PR_{g_4}(x) \to \PR_{g_4}(y) \bigr)$. 
\end{proof}

\begin{cl}\label{Cl:g_4_box}
Let $i, l \in W$. 
\begin{enumerate}
	\item $\PA \vdash S(\num{i}) \to \PR_{g_4} \Bigl( \gn{\bigvee_{i \prec j} S(\num{j})} \Bigr)$. 
	\item If $i \prec l$, then $\PA \vdash S(\num{i}) \to \neg \PR_{g_4}(\gn{\neg S(\num{l})})$.
\end{enumerate}
\end{cl}
\begin{proof}
We proceed in $\PA + S(\num{i})$: 
Let $m$ be such that $h(m) = 0$ and $h(m+1) = i$. 

1. Since $\bigvee_{i \prec j} S(\num{j})$ is a t.c.~of $P_{T, m-1} \cup \bigl\{\bigvee_{i \prec j} S(\num{j}) \bigr\}$, it follows that $\PR_{g_4} \Bigl( \gn{\bigvee_{i \prec j} S(\num{j})} \Bigr)$ holds by Claim \ref{Cl:g_4_eq2}. 

2. Suppose, towards a contradiction, that $\neg S(\num{l})$ is a t.c.~of $P_{T, m-1} \cup \bigl\{\bigvee_{i \prec j} S(\num{j}) \bigr\}$. 
Then, $\bigvee_{i \prec j} S(\num{j}) \to \neg S(\num{l})$ is a t.c.~of $P_{T, m-1}$. 
Since $S(\num{l})$ is a disjunct of $\bigvee_{i \prec j} S(\num{j})$, $S(\num{l}) \to \neg S(\num{l})$ is also a t.c.~of $P_{T, m-1}$. 
We have that $\neg S(\num{l})$ is a t.c.~of $P_{T, m-1}$. 
This is a contradiction. 
Hence, $\neg S(\num{l})$ is not a t.c.~of $P_{T, m-1} \cup \bigl\{\bigvee_{i \prec j} S(\num{j}) \bigr\}$. 
By Claim \ref{Cl:g_4_eq2}, $\neg \PR_{g_4}(\gn{\neg S(\num{l})})$ holds. 
\end{proof}

The following claim is proved in the same fashion as in the usual proof of Solovay's arithmetical completeness theorem by using Claim \ref{Cl:g_4_box}.

\begin{cl}\label{Cl:g_4}
Let $i \in W$ and $B \in \MF$. 
\begin{enumerate}
	\item If $i \Vdash B$, then $\PA \vdash S(\num{i}) \to f_{g_4}(B)$. 
	\item If $i \nVdash B$, then $\PA \vdash S(\num{i}) \to \neg f_{g_4}(B)$. 
\end{enumerate}
\end{cl}

The first clause of Theorem \ref{Thm:K} follows from Claims \ref{Cl:g_4_eq} and \ref{Cl:g_4_2}. 
The second clause follows from Proposition \ref{Prop:h}.3 and Claim \ref{Cl:g_4}. 
\end{proof}

\begin{cor}
\[
	\K = \bigcap \{\PL(\PR_T) \mid \PR_T(x)\ \text{is a}\ \Sigma_1\ \text{provability predicate of}\ T\ \text{satisfying}\ \D{2}\}. 
\]
Moreover, there exists a $\Sigma_1$ provability predicate $\PR_T(x)$ of $T$ such that $\K = \PL(\PR_T)$. 
\end{cor}

\section{Appendix: Interchangeability of $\Box$ and $\Diamond$ in $\NR$}\label{Sec:NR_Box}

The language of propositional modal logic does not have the symbol $\Diamond$ as a modal operator. 
We introduce the expression $\Diamond A$ as the abbreviation of $\neg \Box \neg A$. 
However, in the logic $\N$, $\Box$ and $\Diamond$ are not dual operators, that is, we show that $\neg \Diamond \neg p \leftrightarrow \Box p$ is not provable in $\NRF$. 

\begin{prop}
$\NRF \nvdash \neg \Diamond \neg p \to \Box p$. 
\end{prop}
\begin{proof}
Let $\mathcal{F} = (\{a, b\}, \{\prec_B\}_{B \in \MF})$ be the $\N$-frame defined as follows: for each $B \in \MF$ and $x, y \in \{a, b\}$, 
\[
	x \prec_B y : \iff \begin{cases} y = b & \text{if}\ B \equiv p, \\ x = y & \text{otherwise.}\end{cases}
\]
Obvioulsy, $\mathcal{F}$ is serial. 
We prove that $\mathcal{F}$ is transitive. 
Suppose $x \prec_{\Box B} y \prec_{B} z$ for $B \in \MF$ and $x, y, z \in \{a, b\}$. 
Since $\Box B \not \equiv p$, we have $x = y$. 
Thus, we obtain $x \prec_B z$ because $y \prec_B z$. 

Let $\Vdash$ be a satisfaction relation on $\mathcal{F}$ such that $a \Vdash p$ and $b \nVdash p$. 
Since $\neg \neg p \not \equiv p$, we have that $a \prec_{\neg \neg p} x$ if and only of $x = a$. 
So, we obtain $a \Vdash \Box \neg \neg p$ because $a \Vdash \neg \neg p$. 
Since $a \Vdash \neg \neg \Box \neg \neg p$, we get $a \Vdash \neg \Diamond \neg p$. 
On the other hand, we have $a \nVdash \Box p$ because $a \prec_p b$ and $b \nVdash p$. 
Therefore, $\neg \Diamond \neg p \to \Box p$ is not valid in $\mathcal{F}$. 
By Corollary \ref{Cor:NF}, we conclude that $\NRF \nvdash \neg \Diamond \neg p \to \Box p$. 
\end{proof}

From another point of view, in $\NR$, the operators $\Box$ and $\Diamond$ have an interesting relationship.
That is, in a sense, $\Box$ and $\Diamond$ are interchangeable in $\NR$.
To state this fact precisely, we introduce the following translation $\chi$.

\begin{defn}
We define a translation $\chi$ of $\LB$-formulas recursively as follows: 
\begin{enumerate}
	\item $\chi(A)$ is $A$ if $A$ is a propositional variable or $\bot$, 
	\item $\chi(\neg A)$ is $\neg \chi(A)$, 
	\item $\chi(A \circ B)$ is $\chi(A) \circ \chi(B)$ for $\circ \in \{\land, \lor, \to\}$, 
	\item $\chi(\Box A)$ is $\Diamond \chi(A)$. 
\end{enumerate}
\end{defn}

That is, $\chi(A)$ is obtained from $A$ by replacing every $\Box$ with $\Diamond$. 

\begin{prop}
For any $A \in \MF$, $\NR \vdash A$ if and only if $\NR \vdash \chi(A)$. 
\end{prop}
\begin{proof}
$(\Rightarrow)$: We prove this implication by induction on the length of proofs in $\NR$. 
It suffices to prove that $\NR$ is closed under the rules $\dfrac{A}{\Diamond A}$ and $\dfrac{\neg A}{\neg \Diamond A}$. 
Suppose $\NR \vdash A$. 
Then, $\NR \vdash \neg \neg A$. 
By the rule \textsc{Ros}, $\NR \vdash \neg \Box \neg A$, that is, $\NR \vdash \Diamond A$. 

Suppose $\NR \vdash \neg A$. 
By \textsc{Nec}, $\NR \vdash \Box \neg A$, and then $\NR \vdash \neg \neg \Box \neg A$. 
This means $\NR \vdash \neg \Diamond A$. 

$(\Leftarrow)$: We prove the contrapositive. 
Suppose $\NR \nvdash A$. 
Then, by Theorem \ref{Thm:complNR}, there exists a serial $\N$-model $\mathcal{M} = (W, \{\prec_B\}_{B \in \MF}, \Vdash)$ and $w \in W$ such that $w \nVdash A$. 
For each $B \in \MF$, let $\prec_B^*$ be the binary relation on $W$ defined as follows:
\[
	\prec_B^* : = \begin{cases} \prec_D & \text{if}\ B\ \text{is of the form}\ \neg \neg \chi(\chi(D)), \\
			\prec_B & \text{otherwise}.
			\end{cases}
\]
Let $\mathcal{M}^*$ be the $\N$-model $(W, \{\prec_B^*\}_{B \in \MF}, \Vdash^*)$ defined by $x \Vdash^* p : \iff x \Vdash p$. 
It is easy to see that the frame of $\mathcal{M}^*$ is also serial. 

\begin{cl}\label{Cl:NR}
For any $\LB$-formula $C$ and $x \in W$, $x \Vdash^* \chi(\chi(C))$ if and only if $x \Vdash C$. 
\end{cl}
\begin{proof}
This claim is proved by induction on the construction of $C$. 
If $C$ is a propositional variable, the claim is trivial because $\chi(\chi(p))$ is exactly $p$. 
The cases of $\bot$ and propositional connectives are easy. 

We prove the case that $C$ is of the form $\Box D$. 
Notice $\prec_{\neg \neg \chi(\chi(D))}^* = \prec_D$. 
By the induction hypothesis, for any $y \in W$, $y \Vdash^* \chi(\chi(D))$ if and only if $y \Vdash D$. 
\begin{align*}
	x \Vdash^* \chi(\chi(\Box D)) & \iff x \Vdash^* \Box \neg \neg \chi(\chi(D)), \\
		& \iff \forall y \in W \bigl(x \prec_{\neg \neg \chi(\chi(D))}^* y \Rightarrow y \Vdash^* \neg \neg \chi(\chi(D)) \bigr), \\
		& \iff \forall y \in W \bigl(x \prec_D y \Rightarrow y \Vdash^* \chi(\chi(D)) \bigr), \\
		& \iff \forall y \in W \bigl(x \prec_D y \Rightarrow y \Vdash D \bigr), \\
		& \iff x \Vdash \Box D. \tag*{\mbox{\qedhere}}
\end{align*}
\end{proof}

Since $w \nVdash A$, we obtain $w \nVdash^* \chi(\chi(A))$ by Claim \ref{Cl:NR}. 
By Theorem \ref{Thm:complNR} again, we obtain $\NR \nvdash \chi(\chi(A))$. 
Since we have already proved the implication $(\Rightarrow)$ of the proposition, we conclude $\NR \nvdash \chi(A)$. 
\end{proof}

\end{document}